\documentclass[12pt]{article}
\usepackage[utf8]{inputenc}
\usepackage{amssymb,amsfonts,amsmath,amsthm}
\usepackage{mathrsfs}
\usepackage{mathtools}
\usepackage{color}
\usepackage[colorlinks = true,
linkcolor = red,
urlcolor  = red,
citecolor = blue,
anchorcolor = red]{hyperref}

\numberwithin{equation}{section}

\newtheorem{theorem}{Theorem}[section]
\newtheorem{lemma}[theorem]{Lemma}

\newtheorem{remark}[theorem]{Remark}

\usepackage{soul}
\usepackage{xcolor}
\sethlcolor{pink}
\begin{document}
\title{\bf Stability and Blow-Up for a Suspension Bridge Plate Model with Fractional Damping and Memory}
	\author{
		\small{\bf Iqra Kanwal}
        \footnote{Corresponding author. Email:ikanwal1703@gmail.com} \\
		\small School of Mathematics and Statistics, Shanxi University, Taiyuan, China \\
		\small {\bf Jianghao Hao}  \footnote{ Email:hjhao@sxu.edu.cn }\\
		\small School of Mathematics and Statistics, Shanxi University, Taiyuan, China,\\
		\small{\bf Muhammad Fahim Aslam}
        \footnote {Email: fahim.sihaab@gmail.com}\\
		\small Department of Mathematics, University of Kotli Azad Jammu and Kashmir(UOKAJK), \\ \small Kotli, Pakistan  \\
        \small{\bf Zayd Hajjej} \footnote{Email: zhajjej@ksu.edu.sa}\\
        \small Department of Mathematics, College of Science, King Saud University,\\
        \small P.O. Box 2455, Riyadh 11451, Saudi Arabia. \\
        \small{\bf Mauricio Sep\'ulveda-Cort\'es} \footnote{Email: mauricio@ing-mat.udec.cl}\\
        \small Departamento de Ingenier\'ia Matem\'atica \\ \small \& Centro de Investigaci\'on en Ingenier\'ia Matem\'atica, Universidad de Concepci\'on,\\
        \small Concepci\'on, Chile. \\
        \small{\bf Rodrigo Vejar-Asem} \footnote{Email: rodrigo.vejar@userena.cl}\\
        \small Departamento de Matem\'aticas, Universidad de La Serena,\\
        \small La Serena, Chile.
        }
	\date{}
	\maketitle
	\begin{abstract}
We investigate a suspension bridge model described by a nonlinear plate equation
incorporating internal fractional damping and infinite memory effects. The system
also includes a nonlinear source term that may induce instability. Using semigroup
theory, we first establish the local well-posedness of solutions in an appropriate
energy space. We then derive conditions ensuring global existence and exponential
stability of solutions. In contrast, when the initial energy is negative, we prove
that solutions blow up in finite time, revealing a threshold phenomenon governing
the long-term dynamics of the system.

To complement the analytical results, we construct a numerical approximation based
on Summation-By-Parts finite differences with Simultaneous Approximation Terms
(SBP--SAT) for the spatial discretization and a Newmark scheme for time integration.
The scheme preserves the structural properties of the continuous energy framework.
Numerical experiments illustrate the stability and blow-up regimes predicted by
the theoretical analysis.
\end{abstract}
\section{ Introduction}
Bridges have always been employed to reach areas that are inaccessible by rivers or narrow valleys, or thickly-settled cities. These buildings usually comprise of a roadway-deck which was held up by cables at each end and crossed over high towers. The arrangement permits lengthy spans, good distribution of loads and a fine architectural design.

Suspension bridges are vulnerable to a variety of dynamic disturbances notwithstanding these benefits. The building may vibrate or oscillate due to wind, traffic loads, and seismic pressures. When these stimuli interact with the bridge's inherent frequencies, they could cause severe oscillatory reactions and, in severe circumstances, jeopardize structural safety or serviceability. As a result, managing suspension bridges' dynamic response is crucial for both mathematical modeling and engineering design.

The sustained may be caused by wind effects, including turbulence and shedding of vortices 
oscillations. Another source of time-dependent forcing is the traffic loads that can be coupled to the vibration modes of the bridge deck. In addition,
seismic events can cause sudden and sharp excitements that actuate the system
far from equilibrium. Provided that these phenomena are not properly regulated, the
dynamics instability can cause threat to the safety and long-term durability
of the structure.

Engineers use different damping mechanisms to reduce the effects.
Damping will decrease the magnitude of vibrations and decrease the possibility of resonance eliminating a portion of the vibrational energy. Conventional machines are dependent on material
dissipation or frictional mechanisms. More recently, strategies founded on
proposals have been made of fractional calculus, such as fractional damping or
fractional-order delay, that are more flexible and more accurate descriptions of
energy dissipation.

Fractional damping can reproduce memory-type responses observed in real
structures and capture a broader range of frequency-dependent behaviors than
classical integer-order damping models. Fractional operators have been
successfully applied in many areas of science and engineering, including nuclear
physics, quantum mechanics, and fluid dynamics (see, for example,
\cite{MainBo,Pod,TorBa}). In structural mechanics, they offer a natural framework
for describing viscoelastic or nonlocal effects, since fractional derivatives
incorporate past states of the system into the present evolution. This property
makes them particularly suitable for modeling distributed internal forces and
history-dependent responses in bridge components (see also \cite{RosSh}).

In this study, we examine a suspension-bridge model with $d\ll\pi$ specified on the domain $\Omega=(0,\pi)\times(-d,d)$. The system includes a nonlinear source, an infinite-memory term, and frictional damping. Our goal is to find conditions that guarantee stability, examine the long-term behavior of solutions, and ascertain the conditions that could lead to finite-time blow-up.

We examine the following plate equation:
\begin{equation} \label{mprobl}
\hspace{-0.7cm}\left\{\begin{array}{lr}
u_{t t}(x, y, t)+\Delta^{2} u(x, y, t)+a_{1} \partial_{t}^{\alpha, \beta} u(x, y, t) & \\[1em]- \int_{0}^{+\infty} g(s) \Delta^2u (t - s) ds=u|u|^{p-2}, & \text { in } \Omega \times(0,+\infty), \\[1em]
u(0, y, t)=u_{x x}(0, y, t)=u(\pi, y, t)=u_{x x}(\pi, y, t)=0, & (y, t) \in(-d, d) \times(0,+\infty), \\[1em]
u_{y y}(x, \pm d, t)+\sigma u_{x x}(x, \pm d, t)=0, & (x, t) \in(0, \pi) \times(0,+\infty), \\[1em]
u_{y y y}(x, \pm d, t)+(2-\sigma) u_{x x y}(x, \pm d, t)=0, & (x, t) \in(0, \pi) \times(0,+\infty), \\[1em]
u(x, y, 0)=u_{0}(x, y), u_{t}(x, y, 0)=u_{1}(x, y), & \text { in } \Omega,\\[1em]
u(x, y, -t)=h(x,y, t),& \text { in } \Omega\times(0,+\infty),
\end{array}\right.
\end{equation}
where $a_1 >0$,  $p > 2$, and $g$ is a function that will be specified later.  For physical reasons, the Poisson ratio $\sigma$ generally lies in the range $(-1, \frac{1}{2})$. Since its value is about $0.3$ for metals and lies between $0.1$ and $0.2$ for concrete, it is reasonable to assume that $0 <\sigma <\frac{1}{2}$.  The modified Caputo's fractional derivative, represented by the symbol $\partial_{t}^{\alpha,\beta}$, is defined as follows (see~\cite{BC,CM}):
$$
 \partial_{t}^{\alpha,\beta} u(t) = \frac{1}{\Gamma(1-\alpha)} \int_{0}^{t} (t-s)^{-\alpha} e^{-\beta(t-s)} u_{s}(s) ds, \hspace{7mm}0 < \alpha < 1,\; \beta>0.
$$
The initial findings about suspension bridges come from McKenna and Walter \cite{Meck2} and McKenna et al. \cite{Meck1}, who demonstrated the existence of nonlinear oscillations and presented a model showing the dynamics of a suspension bridge.The asymptotic dynamics and global attractors for coupled suspension bridge equations were studied by Bochicchio et al. \cite{Vuk} and Ma and Zhong \cite{Ma}, respectively.  Ferrero and Gazzola recently presented a new suspension bridge type employing a plate \cite{Ferr1}. More information about suspension bridge models can be found at \cite{Gazzola}.  The model in \cite{Ferr1} had its bending and stretching energies analyzed in depth in \cite{Algwaiz}.  In their subsequent work, Berchio et al. \cite{Berchio} examined the structural instability of suspension bridge models that incorporated nonlinear plates. The finite time blow-up and uniform stability of suspension bridges have been the subject of numerous recent studies. In \cite{Wang}, Wang  considered the following problem
$$v_{tt}+\Delta^2 v +a_0 v_t+a v=v\vert v\vert^{p-2},$$
with the same boundary conditions as in \eqref{mprobl} and where $a=a(x,y,t)$ is a sign-changing and bounded measurable function. The author established necessary and sufficient conditions for the uniqueness and existence of global solutions as well as the finite time blow-up of these solutions. Next, by considering a nonlinear damping (of the form $\vert v_t\vert^{m-2}v_t$) instead of a linear one,  Liu et al. \cite{Liu}  extended the work of Wang \cite{Wang}. In \cite{Mess1}, the authors considered system \eqref{mprobl} in the case where $a_1=0$ and with a general source term of the form $h(v)$. Through the use of multiplier techniques, the authors proved an exponential decay rate of energy. It is also worth mentioning the works \cite{Haj1, Haj2, Haj3, Mess2}, which focus on suspension bridges and present other types of damping, including structural and viscoelastic damping.\\
Infinite memory or past history plays a crucial role in the mathematical modeling of various physical systems, particularly in the context of wave equations and viscoelastic materials. In the study of infinite memory wave equations, the system's future state depends not only on its current state but also on its entire past history, often represented through integral terms involving memory kernels. Such models are significant because they capture the hereditary effects and long-term dependencies inherent in many materials and processes. For example, recent research has established exponential stability results for nonlinear wave equations with infinite memory and frictional damping, showing that the infinite memory effect alone can guarantee the system's exponential stability under certain conditions. This highlights the strong influence of past states on the system's energy decay and overall behavior, which cannot be captured by models lacking memory terms. In the presence of infinite memory,many researchers have discussed stability and blow up results of the systems. For this we recommend \cite{FA1,DA,KR,FA2} for the readers for better understanding.\\
Several related studies deal with stability questions for models with fractional
damping. In particular, the works \cite{Akil1, Akil2, Benaissa2, Benaissa1, ZH, FA3}
examine wave equations and Timoshenko-type systems where fractional damping plays a
central role in controlling the evolution of solutions.

Infinite-memory effects are also important in the analysis of suspension-bridge
models, especially when the deck is described by plate-type equations. In such
settings, the response of the structure at any given moment depends not only on its
current state but also on its entire deformation history. This feature is typically
incorporated through integral terms involving memory kernels that represent
viscoelastic and hereditary contributions. Including these terms leads to models that
better capture the long-term dynamic behavior of bridge decks and the way energy is
dissipated within the structure.

From a mechanical point of view, the presence of memory increases the damping
capability of the system, helping to control oscillations produced by external forces
such as traffic loads or wind excitation. Mathematically, studies of nonlinear plate
equations with infinite-memory terms have established well-posedness and various
stability results, showing that memory effects can contribute significantly to
stabilizing the system. These observations are particularly relevant for the design
and assessment of suspension bridges, where long-term durability and resistance to
dynamic instabilities are essential.

The purpose of the present work is to investigate the qualitative behavior of
solutions to the suspension bridge model \eqref{mprobl}, which incorporates
both fractional damping and an infinite-memory term. The interaction between
these two mechanisms introduces additional mathematical difficulties and
leads to a rich dynamical behavior, and creates a competition between dissipative mechanisms and nonlinear
instability, which motivates the analysis carried out in this work.
The main contributions of this paper can be summarized as follows:

\begin{itemize}
\item We give the local well-posedness result for the system in an appropriate energy space, which follows from semi-group theory.

\item Under suitable conditions on the initial data, we prove the global
existence and exponential stability of solutions.

\item When the initial energy is negative, we show that the corresponding
solutions blow up in finite time, revealing a threshold phenomenon in the
long-term dynamics of the model.

\item In order to complement the theoretical analysis, we develop a numerical
approximation based on SBP--SAT finite differences for the spatial
discretization combined with a Newmark time integration scheme, which
preserves the structural properties of the continuous energy framework.
\end{itemize}

 This~paper is structured as follows. Section 2 provides the assumptions and tools necessary to establish the main results. In~Section 3, we prove the well-posedness of our problem. Section 4 is devoted to the study of exponential stability. In Section 5, we establish the finite-time blow-up of a specific solution. Section 6 is devoted to the presentation of numerical simulations. Section 7 presents our conclusions and future works.
\section{Preliminaries}
In this section, we introduce the functional setting and several auxiliary
results that will be used throughout the paper.

We denote by $\langle\cdot,\cdot\rangle$
and $\|\cdot\|$ the inner product and the norm in $L^2(\Omega)$,
respectively.

We define the space
\[
H^2_*(\Omega)=\left\{u\in H^2(\Omega): u=0 \ \text{on}\ \{0,\pi\}\times(-d,d)\right\},
\]
equipped with the scalar product
\[
(u,v)_{H^2_*(\Omega)}=\int_\Omega
\big[\Delta u\,\Delta v
+(1-\sigma)(2u_{xy}v_{xy}-u_{xx}v_{yy}-u_{yy}v_{xx})\big]\,dx\,dy.
\]
Note that $(H^2_*(\Omega),(\cdot,\cdot)_{H^2_*(\Omega)})$ is a Hilbert space and
that the induced norm $\|\cdot\|_{H^2_*(\Omega)}$ is equivalent to the usual
$H^2$-norm (see \cite{Ferr1}).

We then have the following Lemmas.\\[0pt]
\begin{lemma}\cite{Ferr1} If $0<\sigma<\frac{1}{2}$ and $f \in L^{2}(\Omega)$, then there is a unique $v \in H^2_{*}(\Omega)$ such that, for all $u \in H^2_{*}(\Omega)$, we have
\begin{equation} \label{eq2.1}
(v, u)_{H^2_{*}(\Omega)}=\int_{\Omega} f v.
\end{equation}
The function $u \in H^2_{*}(\Omega)$ satisfying \eqref{eq2.1} is known as the weak solution to the stationary problem
\begin{equation}
\left\{\begin{array}{l}
\Delta^{2} u=f \\
u(0, y)=u(\pi, y)=u_{x x}(0, y)=u_{x x}(\pi, y)=0 \\
u_{y y}(x, \pm d, t)+\sigma u_{x x}(x, \pm d, t)=u_{y y y}(x, \pm d)+(2-\sigma) u_{x x y}(x, \pm d)=0
\end{array}\right.
\end{equation}
\end{lemma}
\begin{lemma} \cite{Wang} \label{lem2.2} Let $u \in H^2_{*}(\Omega)$ and $1 \leq r<+\infty$. Then, we have
\begin{equation}\label{eq2.3}
\|u\|_{r}^{r} \leq C_{e}\|u\|_{H^2_{*}(\Omega)}^{r},
\end{equation}
for some positive constant $C_{e}=C_{e}(\Omega, r)$, where $\|\cdot\|_{r}$ is the usual $L^{r}(\Omega)$-norm.\\[0pt]
\end{lemma}
\begin{lemma}\cite{BM}\label{le2} Let $\xi$ be the function
$$
\xi(\vartheta)=|\vartheta|^{\frac{2 \alpha-1}{2}}, \vartheta \in \mathbb{R}, 0<\alpha<1 .
$$
Then, the relationship between the 'input' $U$ and the 'output' $O$ of the system
\begin{equation}
\left\{\begin{array}{l}
\phi_t(x, y, \vartheta, t)+\left(\vartheta^{2}+\beta\right) \phi(x, y, \vartheta, t)-U(x, y, t) \xi(\vartheta)=0, \vartheta \in \mathbb{R}, t>0, \beta>0,\\
\phi(x, y, \vartheta, 0)=0, \\
O(t)=\pi^{-1} \sin (\alpha \pi) \displaystyle\int_{-\infty}^{+\infty} \phi(x, y, \vartheta, t) \xi(\vartheta) d \vartheta
\end{array}\right.
\end{equation}
is given by
$$
O=I^{1-\alpha, \beta} U,
$$
where
$$
I^{\alpha, \beta} u(t)=\frac{1}{\Gamma(\alpha)} \int_{0}^{t}(t-s)^{\alpha-1} e^{-\beta(t-s)} u(s) d s .
$$
\end{lemma}
We also need the next lemma.\\[0pt]
\begin{lemma}\cite{Pod} If $\left.\lambda \in D_{\beta}=\mathbb{C} \backslash\right]-\infty,-\beta[$, then
$$
\int_{-\infty}^{+\infty} \frac{\xi^{2}(\vartheta)}{\lambda+\beta+\vartheta^{2}} d \vartheta=\frac{\pi}{\sin (\alpha \pi)}(\lambda+\beta)^{\alpha-1} .
$$
\end{lemma}
We impose the following assumptions on the memory kernel $g$:\\
(H1) $g : [0, \infty) \to (0, \infty)$ is a $L^1([0, \infty))\cap C^1([0, \infty))$ function  such that

$$g(0) > 0, \quad  \int_{0}^{\infty} g(s)\,ds = 1 - \lambda > 0.$$
(H2) There exist positive constants $c_0$ and $c_1$ such that
$$-c_0 g(t)\leq g'(t) \leq -c_1 g(t),\;\; \forall\;t\geq 0.$$
Following  \cite{RV2, RV1}, we introduce the relative history variable:
\begin{equation}
    \mu (x,y,t,s) = u(x,y,t)-u(x,y,t-s),\;(x,y)\in\Omega,\; t, s>0.
    \label{eq12}
\end{equation}
The variable $\mu$ represents the relative history of $u$ that satisfies the following equation:
\begin{equation*}\label{2.3m}
    \mu_t (x,y,t,s) - u_t (x,y,t) + \mu_s (x,y,s) = 0,\; (x,y)\in \Omega, \hspace{2mm} t,s > 0,\;\text{and}\;\mu(x,y,t,0)=0.\;
\end{equation*}
{\color{black}{Introducing the following boundary conditions:
\begin{eqnarray}\label{bc}
\left\{
\begin{array}{lcr}
u_{xx}(0,y)=u_{xx}(\pi,y)=0, \\

u_{yy}(x,\pm d)+\sigma u_{xx}(x,\pm d)=0, \\

u_{yyy}(x,\pm d)+(2-\sigma)u_{xxy}(x,\pm d)=0.
\end{array}
\right.
\end{eqnarray}
}}
By using Lemma \eqref{le2} and equation \eqref{eq12}, the system \eqref{mprobl} can be expressed in the following manner:
\vspace{-1 mm}
\begin{equation} \label{eq2.5}
\hspace{-0.4cm}\left\{\begin{array}{lr}
u_{t t}(x, y, t)+\lambda\Delta^{2} u(x, y, t)+ \displaystyle\int_{0}^{+\infty} g(s) \Delta^2 \mu (x,y,t,s) ds \\[1em]
+\kappa \displaystyle\int_{-\infty}^{+\infty} \phi(x, y, \vartheta, t) \xi(\vartheta) d \vartheta=u|u|^{p-2},  & (x, y) \in \Omega, t>0, \\[1em]
 \phi_{t}(x, y, \vartheta, t)+\left(\vartheta^{2}+\beta\right) \phi(x, y, \vartheta, t)-u_t (x,y,t) \xi(\vartheta)=0, & (x, y) \in \Omega, \vartheta \in \mathbb{R}, t>0, \\[1em]
\mu_t (x,y,t,s) + \mu_s(x,y,t,s)=u_t(x,y,t),& (x, y) \in \Omega,\; s, t>0, \\[1em]
u(0, y, t)=u_{x x}(0, y, t)=u(\pi, y, t)=u_{x x}(\pi, y, t)=0, & (x, t) \in(0, \pi) \times(0,+\infty), \\[1em]
u_{y y}(x, \pm d, t)+\sigma u_{x x}(x, \pm d, t)=0, & (x, t) \in(0, \pi) \times(0,+\infty), \\[1em]
u_{y y y}(x, \pm d, t)+(2-\sigma) u_{x x y}(x, \pm d, t)=0, & (x, y) \in \Omega, t>0, \\[1em]
u(x, y, 0)=u_{0}(x, y), u_{t}(x, y, 0)=u_{1}(x, y), & (x, y) \in \Omega, \\[1em]
\phi(x, y, \vartheta, 0)=0, & (x, y) \in \Omega, \vartheta \in \mathbb{R}, \\[1em]
\mu(x,y,0) = 0, \quad \mu_0 (x,y,0,s) = u_0 (x) - u_0 (x,-s),   & (x, y) \in \Omega, s>0,
\end{array}\right.
\end{equation}
where $\kappa=\frac{a_{1} \sin (\alpha \pi)}{\pi}$.
\section{Well-Posedness}
In this section, we will prove the well-posedness of system \eqref{eq2.5}.
Let $\mathcal{H}$ be the energy space which is given by
$$
\mathcal{H}=H^2_{*}(\Omega) \times L^{2}(\Omega) \times L^{2}(\Omega \times \mathbb{R}) \times L_g^2 (\mathbb{R}_+, H^2_{*} (\Omega)),
$$
where
$$
L_g^2 (\mathbb{R}_+, H^2_{*} (\Omega)) = \Biggl\{\omega : \mathbb{R}_+ \to H^2_{*} (\Omega), \int_{0}^{+\infty} g(s) ||\omega(s)||^2_{H^2_{*}} ds < \infty    \Biggl\};
$$
the space $L_g^2 (\mathbb{R}_+, H^2_{*})$ is endowed with the inner product:
$$
\bigl \langle w_1,w_2 \bigr \rangle_{L_g^2 (\mathbb{R}_+, H^2_{*} (\Omega)} =  \int_{0}^{+\infty} g(s) ( w_1 (s),  w_2 (s))_{H^2_{*}} dx dy ds.
$$
Let $v=u_t$ and $U = (u,v,\phi,\mu)^T.$ So, problem \eqref{eq2.5} can be written as:\\
 \vspace{4mm}
 \hspace{20mm}
 \begin{equation}
 \quad \quad \quad  \left\{
\begin{array}{lr}
U_t(t) + AU(t) = J(U(t))= (0, |u|^{p-2} u, 0, 0)^{T} \\
U(0) = U_0,
\label{A1}
\end{array}
\right.
\end{equation}
\\where the operator ${A}$ is defined by
\vspace{-.5 em}
\begin{equation}\label{eq3.2}
 AU = \left( \begin{array}{c}
-v \\
\vspace{2mm}
-\lambda \Delta^2 u -\int_{0}^{+\infty} g(s) \Delta^2 \mu (x,y,s) d s + b \int_{-\infty}^{+\infty} \phi(x,y,\vartheta) \xi(\vartheta) d\vartheta \\
\vspace{2mm}
(\xi^2 + \beta) \phi - u(x,y) \xi(\vartheta)\\
\vspace{2mm}
\mu_s (s)-u \end{array} \right),
\end{equation}
  with domain
\begin{eqnarray*}
D({A})&=&\Big\{(u, v, \phi, \mu)\in\mathcal{H}: u\in H^4(\Omega)\; \text{satisfying } \eqref{bc},\; v\in H^2_{*}(\Omega),\; \xi \phi\in L^2(\Omega\times\mathbb{R}),\\
 &&\hspace{2.5cm}-({\xi^2} + {\beta})\phi - u(x,y) \xi(\vartheta)\in L^2(\Omega\times\mathbb{R}),\; \mu_{s}\in L_g^2 (\mathbb{R}_+, H^2_{*}), \; \mu(., 0)=0  \Big\}.
\end{eqnarray*}

For all $U=(u, v, \phi, \mu)^{T}$ and $\tilde{U}=(\tilde{u}, \tilde{v}, \tilde{\phi}, \tilde{\mu})^{T} \in \mathcal{H}$, we define the inner product:
$$(U, \tilde{U})_{\mathcal{H}}=\lambda(u, \tilde{u})_{H^2_{*} (\Omega)}+\displaystyle\int_\Omega v \tilde{v} dx dy+\kappa \displaystyle\int_{\Omega} \int_{-\infty}^{+\infty} \phi(x, y, \vartheta) \tilde{\phi}(x, y, \vartheta) d\vartheta dx dy+ \bigl \langle \mu, \tilde{\mu} \bigr \rangle_{L_g^2 (\mathbb{R}_+, H^2_{*} (\Omega) )}.$$
By arguments similar to those in \cite{aslam25, ZH}, one can show that the operator $A$ is maximal monotone, that is $(A U, U)_{\mathcal{H}}\geq 0$  and the operator $Id + A $ is onto. This leads to the following result. 
\begin{theorem}
Suppose that (H1)-(H2) hold. Then, there exists $T>0$ such that if $U_0\in \mathcal{H}$, then system \eqref{eq2.5} has a unique weak solution $U\in C([0, T); \mathcal{H})$. In addition, if $U_0\in D(A)$, then problem \eqref{eq2.5} has a unique strong solution
$U\in C([0, T); D(A))\cap C^1([0, T); \mathcal{H})$.
\end{theorem}
Now, we introduce the functional energy associated with the problem \eqref{eq2.5}  by
\begin{equation}\label{eq2.6}
\begin{split}
E(t)
&= \frac{1}{2} ||u_t||_2^2 + \frac{\kappa}{2} \int_{\Omega} \int_{-\infty}^{+\infty} | \phi(x,y,\vartheta,t)|^2 d\vartheta dxdy + \frac{\lambda}{2} ||u||^2_{H^2_{*}}
\\
&\quad -\frac{1}{p} || u(t)||_p^p +\frac{1}{2} \int_{0}^{+\infty} g(s) ||\mu(s)||^2_{H^2_{*}(\Omega)} ds,
\end{split}
\end{equation}
which  satisfies 
\begin{equation}\label{eq2.7}
E^{\prime}(t)
= \frac{1}{2} \int_{0}^{+\infty} g'(s) ||\mu(s)||^2_{H^2_{*}(\Omega))} ds
 -\kappa \int_{\Omega} \int_{-\infty}^{+\infty} (\vartheta^2 + \beta) |\phi(x,y,\vartheta,t)|^2 d\vartheta dx dy \leq 0.
\end{equation}

Next, we will show the global existence of the solution of problem \eqref{eq2.5}. To do this,  we introduce the following functionals:
\begin{equation}
\begin{aligned}
I(t)= & \lambda\|u(t)\|_{H^2_{*} (\Omega)}^{2}-\|u(t)\|_{p}^{p}+\kappa \int_{\Omega} \int_{-\infty}^{+\infty}|\phi(x,y,\vartheta, t)|^{2} d\vartheta dx dy\\
& +\int_{0}^{+\infty} g(s)\left\|\mu(s)\right\|^{2}_{H^2_{*} (\Omega)} ds,
\end{aligned}
\end{equation}
\begin{equation}
\begin{aligned}
J(t)= & \frac{\lambda}{2}\|u(t)\|_{H^2_{*} (\Omega)}^{2}-\frac{1}{p}\|u(t)\|_{p}^{p}+\frac{\kappa}{2} \int_{\Omega} \int_{-\infty}^{+\infty}|\phi(x,y,\vartheta, t)|^{2} d\vartheta dx dy \\
& +\frac{1}{2} \int_{0}^{+\infty} g(s)\left\|\mu(s)\right\|^{2}_{H^2_{*} (\Omega)} ds
\end{aligned}
\end{equation}

\begin{lemma}\label{lem4.1}
Suppose that (H1)-(H2) hold. Then for any $U_{0} \in \mathcal{H}$ satisfying
\begin{equation}\label{eq4.3}
C_{e}\left(\frac{2 p}{p-2} E(0)\right)^{\frac{p-2}{2}}<\lambda^{\frac{p}{2}},\;\; I(0)>0,
\end{equation}
we have  $I(t)>0, \forall t>0$. Moreover,  the solution of the system \eqref{eq2.5} is bounded and global.
\end{lemma}
\begin{proof}
Since $u$ is continuous and $I\left(0\right)>0$, there exists $T^{*} \leq T$, such that $I(t) \geq 0$, for all $t \in\left[0, T^{*}\right)$.\\
Besides, we have
\begin{eqnarray*}
\frac{2p}{p-2} J(t)&=&\frac{2}{p-2} I(t)+ \lambda\|u(t)\|_{H^2_{*} (\Omega)}^{2}+\kappa \int_{\Omega} \int_{-\infty}^{+\infty}|\phi(x,y,\vartheta, t)|^{2} d\vartheta dx dy\nonumber\\
&&+\int_{0}^{+\infty} g(s)\left\|\mu(s)\right\|^{2}_{H^2_{*} (\Omega)} ds.
\end{eqnarray*}
This implies that
\begin{equation}\label{eq4.4}
\lambda\|u(t)\|_{H^2_{*} (\Omega)}^{2} \leq \frac{2 p}{p-2} J(t) \leq \frac{2 p}{p-2} E(0), \quad \forall t \in\left[0, T^{*}\right)
\end{equation}
Using \eqref{eq2.3}, \eqref{eq4.3} and \eqref{eq4.4},  we infer that
\begin{equation}
\begin{aligned}
\| u(t)\|_{p}^{p} &\leq C_e \|u(t)\|_{H^2_{*} (\Omega)}^{p} \\
& \leq C_{e}\left(\frac{2 p}{\lambda(p-2)} E(0)\right)^{\frac{p-2}{2 }} \| u(t)\|_{H^2_{*} (\Omega)}^{2} \\
&< \lambda \|u(t)\|_{H^2_{*} (\Omega)}^{2}, \quad \forall t \in\left[0, T^{*}\right)
\end{aligned}
\end{equation}
Hence, $\lambda\|u(t)\|_{H^2_{*} (\Omega)}^{2}-\| u(t)\|_{p}^{p}>0, \forall t \in\left[0, T^{*}\right)$. Therefore, we have $I(t)>0,\;\forall\; t \in\left[0, T^{*}\right)$. By repeating this procedure, $T^{*}$ can be extended to $T$.\\
In addition, one can easily see that
$$ \frac{1}{2}\left\|u_{t}(t)\right\|_{2}^{2}+\frac{\lambda(p-2)}{2 p }\|u(t)\|_{H^2_{*} (\Omega)}^{2}\leq \frac{1}{2}\left\|u_{t}(t)\right\|_{2}^{2}+J(t)=E(t)\leq E(0),$$
which implies that the solution of system \eqref{eq2.5} is bounded and global.
\end{proof}
{\color{black}{
\begin{remark}
According to the assumptions of Lemma \ref{lem4.1}, the quantity $\frac{\lambda}{2}\|u(t)\|_{H^2_{*} (\Omega)}^{2}-\frac{1}{p}\|u(t)\|_{p}^{p}$ is positive, and therefore the energy $E(t)$ is positive for all $t\geq 0$.
\end{remark}
}}
\section{Exponential Stability}\label{sec_stability}
In order to establish the energy decay result, let us consider the Lyapunov functional:
\begin{equation}\label{eq5.1}
\mathcal{L}(t)=N E(t)+N_{1} \psi_{1}(t)+N_2\psi_{2}(t),
\end{equation}
where $N, N_{1}$ and $N_{2}$ are positive constants that will be fixed later, and
$$
\begin{aligned}
& \psi_{1}(t)=\int_{\Omega} u u_{t} d x dy+\frac{\kappa}{2}\int_{\Omega} \int_{-\infty}^{+\infty}\left(\vartheta^{2}+\beta\right)\left(\int_{0}^{t} \phi(x,y,\vartheta,s) d s+\frac{u_{0}(x,y) \xi(\vartheta)}{\vartheta^{2}+\beta}\right)^{2} d \vartheta dx dy \\
& \psi_{2}(t)=-\int_{\Omega} u_{t} \int_{0}^{+\infty} g(s) \mu(s) ds dx dy
\end{aligned}
$$
\begin{lemma}\label{lemimpo}(\cite{aslam25})
Let $(u, u_t, \phi, \mu)$ be a regular solution of the problem \eqref{eq2.5}. Then, we have
$$
\begin{aligned}
& \int_{\Omega} \int_{-\infty}^{+\infty}\left(\vartheta^{2}+\beta\right) \phi(x,y,\vartheta,t) \int_{0}^{t} \phi(x,y,\vartheta,s) ds d\vartheta dx dy \\
= & \int_{\Omega} u(x,y,t) \int_{-\infty}^{+\infty} \phi(x,y,\vartheta,t) \xi(\vartheta) d\vartheta dx dy \\
& -\int_{\Omega} u_{0}(x,y) \int_{-\infty}^{+\infty} \phi(x,y,\vartheta,t) \xi(\vartheta) d\vartheta dx dy-\int_{\Omega} \int_{-\infty}^{+\infty}|\phi(x,y,\vartheta,t)|^{2} d\vartheta dx dy.
\end{aligned}
$$
holds.
\end{lemma}
\begin{lemma}
Let $(u, u_t, \phi, \mu)$ be a solution of problem \eqref{eq2.5}, then there exist two positive constants $\alpha_{1}$ and $\alpha_{2}$ such that
\begin{equation}\label{equ5.4}
\alpha_{1} E(t) \leq \mathcal{L}(t) \leq \alpha_{2} E(t).
\end{equation}
The proof is obvious.
\end{lemma}
\begin{lemma}
Let $(u, u_t, \phi, \mu)$ be the solution of \eqref{eq2.5}, then, we have
\begin{equation}
\begin{aligned}\label{equ5.5}
\psi_{1}^{\prime}(t)\leq & \left\|u_{t}\right\|_{2}^{2}-\frac{\lambda}{2}\|u\|^2_{H^2_*(\Omega)}+\frac{1-\lambda}{2\lambda} \int_{0}^{+\infty} g(s)\left\|\mu(s)\right\|^2_{H^2_*(\Omega)} d s \\
& -\kappa \int_{\Omega} \int_{-\infty}^{+\infty}|\phi(x,y,\vartheta,t)|^2 d\vartheta dx dy+\|u(t)\|_{p}^{p}
\end{aligned}
\end{equation}
\end{lemma}
\begin{proof}
By differentiating $\psi_1(t)$ and using \eqref{eq2.5} and Lemma \ref{lemimpo}, we obtain
\begin{eqnarray}\label{psi1}
\psi^{\prime}_1(t)&=&\|u_{t}\|_{2}^{2}-\lambda \|u\|^2_{H^2_*}- \displaystyle\int_{0}^{+\infty} g(s) \int_{\Omega} \Delta^{2}\mu(s)\,u(x,y,t)\,dx dy\,ds\nonumber\\
&&-\kappa \displaystyle\int_{\Omega} u(x,y,t) \int_{-\infty}^{+\infty}
            \phi(x,y,\vartheta,t)\,\xi(\vartheta)\,d\vartheta  dx dy+ \|u\|_{p}^{p}\nonumber\\
&&+\kappa\int_{\Omega} \int_{-\infty}^{+\infty}\left(\vartheta^{2}+\beta\right)\phi(x,y,\vartheta,t)\int_{0}^{t} \phi(x,y,\vartheta,s) d s d\vartheta dx dy\nonumber\\
&&+\kappa\int_{\Omega} \int_{-\infty}^{+\infty}\left(\vartheta^{2}+\beta\right)\phi(x,y,\vartheta,t)\frac{u_{0}(x,y) \xi(\vartheta)}{\vartheta^{2}+\beta} d\vartheta dx dy\nonumber\\
&=&\|u_{t}\|_{2}^{2}-\lambda \|u\|^2_{H^2_*}- \displaystyle\int_{0}^{+\infty} g(s) \int_{\Omega} \Delta^{2}\mu(s)\,u(x,y,t)\,dx dy\,ds\nonumber\\
&&-\kappa \int_{\Omega} \int_{-\infty}^{+\infty}|\phi(x,y,\vartheta,t)|^2 d\vartheta dx dy+\|u(t)\|_{p}^{p}.
\end{eqnarray}
By using  Cauchy-Schwarz and Young inequalities, we get
\begin{align}\label{psi11}
&-\int_{0}^{+\infty} g(s)\int_{\Omega} \Delta^{2}\mu(s) u(x,y,t) dx dyds\nonumber\\
&= -\int_{0}^{+\infty} g(s) ( \mu(s), u)_{H_*^{2}(\Omega)}\,ds \nonumber\\
&\le \frac{\lambda}{2}\|u\|^2_{H^2_*(\Omega)}
  + \frac{1-\lambda}{2\lambda}
    \int_{0}^{+\infty} g(s)\,\|\mu(s)\|_{H_*^{2}(\Omega)}^{2}\,ds.
\end{align}
Inserting \eqref{psi11} in \eqref{psi1}, we get \eqref{equ5.5}.
\end{proof}
\begin{lemma}\label{lem5.5}
Let $(u, u_t, \phi, \mu)$ be the solution of \eqref{eq2.5}. Then, we have, for any $\delta, \delta_1>0$,
\begin{eqnarray}\label{equ5.7}
\psi^{\prime}_2(t)&\leq&-\frac{1-\lambda}{2}\left\|u_{t}\right\|_{2}^{2}+\left(\frac{C_{e}\delta}{2}\left[\frac{2 p E(0)}{\lambda(p-2) }\right]^{p-2}+\frac{\lambda\delta_1}{2}\right)\| u\|_{H_*^{2}(\Omega)}^{2} \nonumber\\
&&+(1-\lambda)\left(1+\frac{C_e c_0^2 }{2(1-\lambda)}+\frac{C_e}{2\delta}+\frac{\lambda}{2\delta_1}+\frac{AC_e}{2}\right) \int_{0}^{+\infty} g(s)\|\mu(s)\|_{H_*^{2}(\Omega)}^{2} ds\nonumber\\
&&+\frac{\kappa^2}{2} \int_{\Omega}\int_{-\infty}^{+\infty}\left(\vartheta^{2}+\beta\right)|\phi(x,y,\vartheta, t)|^{2} d\vartheta dx dy.
\end{eqnarray}
\end{lemma}
\begin{proof}
Multiplying the first equation of \eqref{eq2.5} by $\displaystyle\int_{0}^{+\infty} g(s) \mu(s) ds$ and integrating by parts over $\Omega$, we obtain
\begin{eqnarray}\label{psi22}
&&\frac{d}{dt}\left( - \int_{\Omega} \int_{0}^{+\infty}g(s)\mu(s)u_{t}ds dx dy \right)\nonumber\\
&=&\lambda \int_{0}^{+\infty}g(s)(\mu(s), u)_{H^2_*(\Omega)}ds- \int_{\Omega} \int_{0}^{+\infty}g(s)\mu_{t}(s)u_{t}ds dx dy \nonumber\\
&&+ \Big\|\int_{0}^{+\infty} g(s)\mu(s)ds\Big\|^2_{H^2_*(\Omega)}+ \kappa \left( \int_{0}^{+\infty} g(s)\mu(s)ds\right)\int_{\Omega} \int_{-\infty}^{+\infty}
\phi(x,y,\vartheta,t)\xi(\vartheta)d\vartheta dx dy \nonumber\\
&&- \left( \int_{0}^{+\infty} g(s)\mu(s)ds \right)\int_{\Omega} u|u|^{p-2}dx dy.
\end{eqnarray}
Besides, multiplying the third equation of \eqref{eq2.5} by  $g(s)u_{t}$ and integrating over  $(0,+\infty)\times \Omega$, we get
\begin{align}\label{psi222}
\int_{0}^{+\infty} \int_{\Omega} g(s)\,\mu_{t}(s)\,u_{t}\,dx\,ds
&= (1-\lambda)\int_{\Omega} |u_{t}|^{2}\,dx
   - \int_{\Omega} \int_{0}^{+\infty} g(s)\,\mu_{t}(s)\,u_{t}\,ds dx dy.
\end{align}
Combining \eqref{psi22} and \eqref{psi222}, it holds that
\begin{eqnarray}\label{psi2222}
\psi^{\prime}_{2}(t)&=& -(1-\lambda)\|u_{t}\|_{2}^{2}
 + \int_{\Omega} \int_{0}^{+\infty} g(s)\mu_{s}(s)u_{t} ds dx dy
 + \lambda \int_{0}^{+\infty}g(s)(\mu(s), u)_{H^2_*(\Omega)} ds \nonumber\\
&&+ \Big\|
     \int_{0}^{+\infty} g(s)\mu(s)ds
   \Big\|^2_{H^2_*(\Omega)}
 + \kappa\displaystyle\int_\Omega\left( \int_{0}^{+\infty} g(s)\mu(s)ds \right)
     \int_{-\infty}^{+\infty}
        \phi(x,y,\vartheta,t)\xi(\vartheta)d\vartheta dx dy \nonumber\\
&&- \left( \int_{0}^{+\infty} g(s)\mu(s)ds \right)
   \int_{\Omega} u|u|^{p-2} dx dy.
\end{eqnarray}
Now, we will estimate the terms on the right-hand side of \eqref{psi2222}.
Integrating by parts and combining Young's inequality and Cauchy-Schwarz inequality, we find
\begin{align}\label{psi2^5}
\int_{0}^{+\infty} \int_{\Omega} g(s)\mu_{s}(s)u_{t}dx dy ds
&= - \int_{0}^{+\infty} \int_{\Omega} g'(s)\mu(s)u_{t}dx dy ds \nonumber\\
&\le \frac{1-\lambda}{2}\left\|u_{t}\right\|_{2}^{2}
   + \frac{c_{0}^{2}}{2(1-\lambda)}
     \int_{\Omega} \left( \int_{0}^{+\infty} g(s)\mu^{2}(s)ds \right)^{2} dx dy \nonumber\\
&\le \frac{1-\lambda}{2}\left\|u_{t}\right\|_{2}^{2}
   + \frac{c_{0}^{2}(1-\lambda)}{2(1-\lambda)}
     \int_{\Omega} \int_{0}^{+\infty} g(s)\,\mu^{2}(s) ds dx dy \nonumber\\
&\le \frac{1-\lambda}{2}\left\|u_{t}\right\|_{2}^{2}
   + \frac{c_{0}^{2}C_{e}}{2}
     \int_{0}^{+\infty} g(s)\|\mu(s)\|^2_{H^2_*(\Omega)} ds.
\end{align}
By the help of Young and Cauchy-Schwarz inequalities, the terms $\lambda \displaystyle\int_{0}^{+\infty}g(s)(\mu(s), u(\cdot,s))_{H^2_*} ds$ and  $\Big\| \displaystyle\int_{0}^{+\infty} g(s) \mu(s)ds \Big\|^2_{H^2_*(\Omega)}$ can be estimated as follows:
\begin{align}\label{psi2^6}
\lambda \int_{0}^{+\infty} g(s)\,
   \langle \mu(s), u\rangle_{H_*^{2}}\,ds
&\le \frac{\lambda\delta_1}{2}\|u\|_{H_*^{2}(\Omega)}^{2}
 + \frac{\lambda(1-\lambda)}{2\delta_1}
   \int_{0}^{+\infty} g(s)\|\mu(s)\|_{H_*^{2}(\Omega)}^{2}ds
\end{align}
and
\begin{align}\label{psi2^7}
\Big\| \int_{0}^{+\infty} g(s)\mu(s) ds
\Big\|^2_{H^2_*(\Omega)}
&\le\left\|\left( \int_{0}^{+\infty} g(s) ds \right)^{1/2}\left( \int_{0}^{+\infty} g(s)\,\mu^{2}(s)\,ds \right)^{1/2}
\right\|_{H_*^{2}(\Omega)}^{2} \\
&\le (1-\lambda)\int_{0}^{+\infty} g(s)
     \|\mu(s)\|_{H_*^{2}(\Omega)}^{2}ds.
\end{align}
Next, by Cauchy-Schwarz inequality, we have
\begin{align*}
\left| -\int_{-\infty}^{+\infty}
    \phi(x,y,\vartheta,t)\,\xi(\vartheta)\,d\vartheta \right|
&\le
\left( \int_{-\infty}^{+\infty}
       \frac{\xi^2(\vartheta)}{\vartheta^2+\beta}\,d\vartheta \right)^{1/2}
\left( \int_{-\infty}^{+\infty}
       (\vartheta^2+\beta)\,\phi^2(x,y,\vartheta,t)\,d\vartheta \right)^{1/2}.
\end{align*}
Then, Young's inequality yields
\begin{eqnarray}\label{psi2^8}
&&\left| \kappa\displaystyle\int_\Omega\left( \int_{0}^{+\infty} g(s)\mu(s)ds \right)\int_{-\infty}^{+\infty}\phi(x,y,\vartheta,t)\xi(\vartheta)d\vartheta dx dy  \right|\nonumber\\
&&\le \frac{M}{2} \displaystyle\int_\Omega\left( \int_{0}^{+\infty} g(s)\mu(s)ds \right)^2dx dy+\frac{\kappa^2}{2}\displaystyle\int_\Omega\int_{-\infty}^{+\infty}
       (\vartheta^2+\beta)|\phi(x,y,\vartheta,t)|^2d\vartheta dx dy\nonumber\\
&&\le\frac{MC_e(1-\lambda)}{2}\int_{0}^{+\infty} g(s)
     \|\mu(s)\|_{H_*^{2}(\Omega)}^{2}ds\nonumber\\
&&\hspace{4cm}+\frac{\kappa^2}{2}\displaystyle\int_\Omega\int_{-\infty}^{+\infty}
       (\vartheta^2+\beta)|\phi(x,y,\vartheta,t)|^2d\vartheta dx dy,
\end{eqnarray}
where $$M=\int_{-\infty}^{+\infty}
       \frac{\xi^2(\vartheta)}{\vartheta^2+\beta}\,d\vartheta.$$
Using Young's and Cauchy-Schwarz's inequalities, the last term is estimated as follows:
\begin{eqnarray} \label{psi2^10}
&&-\left( \int_{0}^{+\infty} g(s)\mu(s)ds \right)\int_{\Omega} u|u|^{p-2}dx dy \nonumber\\
&\leq& \frac{\delta}{2}\|u(t)\|_{2p-2}^{2p-2}+\frac{C_e(1-\lambda)}{2\delta}\int_{0}^{+\infty} g(s)
     \|\mu(s)\|_{H_*^{2}(\Omega)}^{2}ds\nonumber\\
&\leq&\frac{\delta C_e}{2}\|u(t)\|_{H^2_*(\Omega)}^{2p-2}+\frac{C_e(1-\lambda)}{2\delta}\int_{0}^{+\infty} g(s)
     \|\mu(s)\|_{H_*^{2}(\Omega)}^{2}ds\nonumber\\
&\leq& \frac{\delta C_{e}}{2}\left[\frac{2 p E(0)}{\lambda(p-2) }\right]^{p-2}\|u(t)\|_{H^2_*(\Omega)}^{2}+\frac{C_e(1-\lambda)}{2\delta}\int_{0}^{+\infty} g(s)\|\mu(s)\|_{H_*^{2}(\Omega)}^{2}ds.
\end{eqnarray}
Inserting the above estimates \eqref{psi2^5}-\eqref{psi2^10} into \eqref{psi2222}, we get the desired estimation.
\end{proof}
\begin{theorem}
\label{thm5.5}
Suppose that the conditions of Lemma \ref{lem4.1} hold. Then there exist positive constants $\theta$ and $\zeta$ such that the global solution of problem \eqref{eq2.5} satisfies
\begin{equation}\label{equ5.14}
E(t) \leq \theta E(0) e^{- \zeta t},\;\forall\;t\geq0.
\end{equation}
\end{theorem}
\begin{proof}
We have
\begin{eqnarray*}
\mathcal{L}^{\prime}(t)&=& N E^{\prime}(t) + N_{1}\psi^{\prime}_{1}(t) + N_{2}\psi^{\prime}_{2}(t)\nonumber\\
&\le& -\kappa \left(N-\frac{N_2\kappa}{2}\right) \int_{\Omega}\int_{-\infty}^{+\infty}
        (\vartheta^{2}+\beta)|\phi(x,y,\vartheta,t)|^2 d\vartheta dx dy \nonumber\\
&-&\left(\frac{N_2(1-\lambda)}{2}- N_{1}\right)\|u_{t}\|_{2}^{2}- \left(\frac{N_{1}\lambda}{2}-N_2\Big(\frac{\lambda \delta_1}{2}+\frac{\delta C_e }{2}\left[\frac{2 p E(0)}{\lambda(p-2) }\right]^{p-2}\Big)\right)|u\|_{H_{*}^{2}(\Omega)}^{2}\nonumber\\
&-&\left(\frac{Nc_1}{2}-\frac{N_1 (1-\lambda)}{2\lambda}-N_2 (1-\lambda)\left(1+\frac{C_e c_0^2 }{2(1-\lambda)}+\frac{C_e}{2\delta}+\frac{\lambda}{2\delta_1}+\frac{AC_e}{2}\right) \right)\int_{0}^{+\infty} g(s)\|\mu(s)\|_{H_{*}^{2}(\Omega)}^{2} ds\nonumber\\
&-&  N_1\kappa\int_{\Omega}\int_{-\infty}^{+\infty}
        |\phi(x,y,\vartheta,t)|^2 d\vartheta dx dy +N_{1}\|u\|_{p}^{p}.
\end{eqnarray*}
Now, we fix $N_1=1$, $\delta_{1} = \frac{1}{4N_{2}}$ and $\delta = \frac{\lambda}{8N_{2}}\left(\frac{1}{\frac{C_{e}}{2}\left( \frac{2p}{\lambda(p-2)}E(0) \right)^{p-2}}\right)$. With these choices, the term $\frac{N_{1}\lambda}{2}-N_2\Big(\frac{\lambda \delta_1}{2}+\frac{\delta C_e }{2}\left[\frac{2 p E(0)}{\lambda(p-2) }\right]^{p-2}\Big)$ is equal to $\frac{\lambda}{4}$.\\
After, we choose $N_{2}$  such that
$$ N_{2}> \frac{2}{1-\lambda}.$$
Finally, we take $N$ large enough such that \eqref{equ5.4} holds true and
$$N>\frac{N_2\kappa}{2},$$
and
$$N>\frac{2 (1-\lambda)}{2\lambda c_1}+\frac{2N_2 (1-\lambda)}{c_1}\left(1+\frac{C_e c_0^2 }{2(1-\lambda)}+\frac{2N_2C^2_e}{\lambda} \left(\frac{2p}{\lambda(p-2)}E(0) \right)^{p-2}+2\lambda N_2+\frac{AC_e}{2}\right).$$
Thus, it holds that
\begin{eqnarray*}
\mathcal{L}^{\prime}(t)\leq -\alpha_3 E(t).
\end{eqnarray*}
Using \eqref{equ5.4} and integrating over $(0,t)$ the last inequality, we get
$$\mathcal{L}(t)\leq \alpha_5 \mathcal{L}(0) e^{-\zeta t},$$
which gives, by using again \eqref{equ5.4}, the estimate \eqref{equ5.14}.
\end{proof}
\section{Blow Up}\label{sec_blowup}
In this section, we will prove that the solution with negative initial energy blows-up in a finite time.
\begin{lemma}\label{le5.1}
There exists $C_1>1$ such that
\begin{equation}
\|u\|_{p}^{s} \leq C_1\left(\|u\|_{p}^{p}+\|u\|^2_{H^2_*(\Omega)}\right)
\end{equation}
\end{lemma}
for any $u \in H_{*}^{2}(\Omega)$ and $2 \leq s \leq p$.\\
The main result of this section reads as follows.
\begin{theorem}\label{teo6.2}
Assume that (H1)-(H2), $E(0)< 0$ and
\begin{equation}\label{eq5.3}
\displaystyle\int_{0}^{\infty} g(s) d s<\frac{p-2}{p},
\end{equation}
then the solution of system \eqref{eq2.5} blows-up in a finite time.
\end{theorem}
\begin{proof}
Let
\begin{equation*}
H(t)= -E(t).
\end{equation*}
By \eqref{eq2.7}, we have
\begin{eqnarray}\label{eq49}
H^{\prime}(t)&=&-E^{\prime}(t)\nonumber\\
&=& -\frac{1}{2} \int_0^{+\infty} g'(s) ||\mu(s)||^2_{H^2_{*}(\Omega)} ds\nonumber\\
&&\hspace{3cm}+\kappa \int_\Omega \int_{-\infty}^{+\infty} (\vartheta^2 + \beta) |\phi(x,y,\vartheta,t)|^2 d\vartheta dx dy \geq 0.
\end{eqnarray}
Thus, we get
\begin{equation}\label{eq5.6}
0 < H(0) \leq H(t) \leq \frac {1}{p}\|u\|_{p}^{p}.
\end{equation}
Let
\begin{equation} \label{eq51}
B(t) = H^{1-\gamma} (t) + \varepsilon \int_\Omega u u_t dx dy,
\end{equation}
where $\varepsilon > 0$ will be specified afterwards, and $\gamma$ is a positive constant that satisfies
\begin{equation}
0 < \gamma < \frac{p-2}{2p}.
\end{equation}
Differentiating (\ref{eq51}) and using \eqref{eq2.5}, we get
\begin{equation}
\begin{split}
B^{\prime}(t)
& =  \varepsilon  ||u_t||_2^2 - \varepsilon \lambda ||u||^2_{H^2_{*}(\Omega)}+\varepsilon\|u\|_{p}^{p}
\\
&\quad +(1 - \gamma) H^{-\gamma}(t) H'(t) -\kappa \varepsilon \int_\Omega u \int_{-\infty}^{+\infty}  \phi(x,y,\vartheta,t)\xi(\vartheta) d\vartheta dx dy
\\
&\quad -\varepsilon \displaystyle\int_{0}^{\infty} g(s) (u, \mu(s))_{H^2_{*}(\Omega)} ds.
\end{split}
\label{eq53}
\end{equation}
As in \eqref{psi11}, we infer that
\begin{equation}\label{eq55}
\int_{0}^{+\infty} g(s) ( \mu(s), u)_{H_*^{2}(\Omega)}ds\le (1-\lambda)\|u\|^2_{H^2_*(\Omega)}
  + \frac{1}{4}\int_{0}^{+\infty} g(s)\,\|\mu(s)\|_{H_*^{2}(\Omega)}^{2}\,ds.
\end{equation}
Substituting (\ref{eq55}) in (\ref{eq53}), we get
\begin{equation}
\begin{split}
B^{\prime}(t)
& \geq (1-\gamma) H^{-\gamma} (t)  H' (t)+ \varepsilon ||u_t||_2^2 -\varepsilon ||u||^2_{H_*^{2}(\Omega)}+\varepsilon\|u\|_{p}^{p}
\\
&\quad -\kappa\varepsilon  \int_\Omega u \int_{-\infty}^{+\infty}  \phi(x,y,\vartheta,t)\xi (\vartheta) d\vartheta dx dy
\\
&\quad -\frac{\varepsilon}{4} \int_{0}^{+\infty} g(s) ||\mu(s)||^2_{H^2_{*}(\Omega)} ds.
\end{split}
\label{eq56}
\end{equation}
By the help of (\ref{eq49}) and Young's inequality, we have, for any $\delta > 0$,
\begin{equation}
\begin{split}
\kappa
& \int_\Omega u \int_{-\infty}^{+\infty}  \phi(x,y,\vartheta,t)\xi (\vartheta) d\vartheta dx dy
\\
&\quad \leq \delta \kappa M ||u||_2^2 + \frac{\kappa}{4\delta}  \int_\Omega \int_{-\infty}^{+\infty} (\vartheta^2 + \beta) |\phi (x,y,\vartheta,t)|^2 d\vartheta dx dy
\\
&\quad \leq \delta \kappa M ||u||_2^2 + \frac{1}{4\delta} H'(t),
\end{split}
\label{eq57}
\end{equation}
Inserting (\ref{eq57}) in (\ref{eq56}), we have
\begin{equation}
\begin{split}
B^{\prime}(t)
& \geq \biggl( (1-\gamma) H^{-\gamma}(t) - \frac{\varepsilon}{4\delta} \biggl)  H' (t)
\\
&\quad + \varepsilon ||u_t||_2^2 - \varepsilon ||u||^2_{H^2_{*}(\Omega)} - \varepsilon \delta \kappa M ||u||_2^2
\\
&\quad - \frac{\varepsilon}{4} \int_{0}^{+\infty} g(s) ||\mu(s)||^2_{H^2_{*}(\Omega)} ds+\varepsilon\|u\|_{p}^{p}.
\end{split}
\label{eq58}
\end{equation}
Next, by selecting $\delta$ such that
\begin{equation}
    \frac{1}{4\delta} = k H^{-\gamma}(t),
    \label{eq59}
\end{equation}
where $k$ is a positive constant to be determined later, (\ref{eq58}) becomes
\begin{equation}
\begin{split}
B^{\prime}(t)
& \geq \bigr[ (1-\gamma) - \varepsilon k \bigr] H^{-\gamma}(t)H'(t) - \varepsilon ||u||^2_{H^2_{*}(\Omega)} - \frac{\varepsilon \kappa M}{4k} H^{\gamma}(t)||u||_2^2\\
&\quad + \varepsilon ||u_t||_2^2+\varepsilon\|u\|_{p}^{p} - \frac{\varepsilon}{4} \int_{0}^{+\infty} g(s) ||\mu(s)||^2_{H^2_{*}(\Omega)} ds.
\label{eq60}
\end{split}
\end{equation}
Using \eqref{eq5.6}, we have
\begin{equation*}
H^{\gamma}(t) \leq \frac{1}{p^{\gamma}}\|u\|_{p}^{p \gamma},
\end{equation*}
which implies that
\begin{equation}\label{eq6Z}
\kappa M H^{\gamma}(t)\|u\|_{2}^{2} \leq C_{2}\|u\|_{p}^{p \gamma+2}
\end{equation}
for some $C_{2}>0$. Putting \eqref{eq6Z} in \eqref{eq60} and using Lemma \ref{le5.1} with $2\leq s=p \gamma+2 \leq p$, we obtain, for any $0<m<1$,
\begin{equation}
\begin{aligned}
B^{\prime}(t) \geq & ((1-\gamma)-\varepsilon k) H^{-\gamma}(t) H^{\prime}(t)+\varepsilon\frac{p(1-m)+2}{2}\left\|u_{t}\right\|_{2}^{2} \\
& +\varepsilon\left(m-\frac{C_{3}}{4 k}\right)\|u\|_{p}^{p}+\varepsilon\left(\frac{p(1-m)\lambda -2}{2}-\frac{C_{3}}{4k}\right)\|u\|^{2}_{H^2_{*}(\Omega)} \\
& +\varepsilon\frac{\kappa p(1-m)}{2}  \int_{\Omega} \int_{-\infty}^{+\infty}|\phi(x,y,\vartheta,t)|^{2} d\vartheta dx dy+\varepsilon p(1-m)H(t)\\
& +\varepsilon\frac{2p(1-m)-1}{4}\int_{0}^{+\infty} g(s)\left\|\mu(s)\right\|^{2}_{H^2_{*}(\Omega)} ds
\end{aligned}
\end{equation}
where $C_{3}=C_1 C_{2}$. \\
First, we choose $0<m<1$ such that
$$m<\min\Big\{1-\frac{1}{2p}, 1-\frac{2}{p\lambda} \Big\}.$$
This choice is justified by the assumption $p>2$ together with \eqref{eq5.3}.

Next, we select $k$ sufficiently large such that
$$m-\frac{C_{3}}{4 k}>0,\;\text{and}\;\; \frac{p(1-m)\lambda -2}{2}-\frac{C_{3}}{4k}>0.$$
Finally, we pick up $\varepsilon$ small enough so that
$$(1-\gamma) - \varepsilon k > 0,\;\;\text{and}\;\;H^{1-\gamma}(0)+\varepsilon \int_{\Omega} u_{0} u_{1} dx dy>0.
$$
Consequently, there exists a positive constant $C_4$ such that
\begin{align}\label{eq5.21}
B^{\prime}(t)&\geq C_4\left\{H(t)+\left\|u_{t}\right\|_{2}^{2}+\|u\|^{2}_{H^2_{*}(\Omega)}+\|u\|_{p}^{p}\right.\nonumber \\
&\quad\;\;\;\;\left.+\int_{\Omega} \int_{-\infty}^{+\infty}|\phi(x,y,\vartheta,t)|^{2} d\vartheta dx dy+\int_{0}^{+\infty} g(s)\left\|\mu(s)\right\|^2_{H^2_{*}(\Omega)} ds\right\}.
\end{align}
and
\begin{equation}\label{GrindEQ__66_}
B(t)\geq B(0)>0,\;\forall\;t>0.
\end{equation}
On the other hand, it is easy to see that
\begin{equation}\label{GrindEQ__67_}
\begin{aligned}
B^{\frac{1}{1-\gamma}}(t)= & \left[H^{1-\gamma}(t)+\varepsilon \int_{\Omega} u u_{t} dx dy \right]^{\frac{1}{1-\gamma}}
\end{aligned}
\end{equation}
\begin{equation}
    \leq C_5\Big[H(t)+\left|\int_{\Omega} u u_{t} dx dy\right|^{\frac{1}{1-\gamma}}\Big]
    \label{eq68}
\end{equation}
for some positive constant $C_5$.\\
Next, employing H\"{o}lder's inequality to obtain
\begin{equation*}\label{eq69}
\left| \int_{\Omega} u u_{t} dx dy \right| \leq\|u\|_{2} \left\|u_{t}\right\|_{2} \leq C_6 \|u\|_{p} \left\|u_{t}\right\|_{2}.
\end{equation*}
Therefore, by using Young's inequality and Lemma \ref{le5.1} with $2<s=\frac{2}{1-2\gamma}<p$, we have
\begin{eqnarray}\label{eq699}
\left|\int_{\Omega} u u_{t} dx dy\right|^{\frac{1}{1-\gamma}} &\leq& C^{\frac{1}{1-\gamma}}_6\|u\|_{p}^{\frac{1}{1-\gamma}} \left\|u_{t}\right\|_{2}^{\frac{1}{1-\gamma}} \nonumber\\
&\leq& C_7\left[\|u\|_{p}^{\frac{2}{1-2\gamma}}+\left\|u_{t}\right\|_{2}^2\right]\nonumber\\
&\leq& C_8 \left[ \|u\|^2_{H^2_*(\Omega)}+\|u\|_{p}^{p}+\left\|u_{t}\right\|_{2}^{2}\right].
\end{eqnarray}
where $C_6, C_7, C_8>0$. Putting \eqref{eq699} in (\ref{eq68}), we get the existence of a positive constant $C_9$ such that
\begin{equation}\label{eq72}
B^{\frac{1}{1-\gamma}}(t) \leq C_9\left[H(t)+\|u\|_{H^2_*(\Omega)}^{2}+\|u\|_{p}^{p}+\|u_{t}\|_{2}^{2}\right].
\end{equation}
From \eqref{eq5.21} and (\ref{eq72}), we deduce that
\begin{equation}
B^{\prime}(t) \geq \varpi B^{\frac{1}{1-\gamma}}(t),\;\forall\;t\geq 0,
\label{eq73}
\end{equation}
for some positive constant $\varpi$.
A simple integration of (\ref{eq73}) gives
$$
B^{\frac{\gamma}{1-\gamma}}(t) \geq \frac{1}{B^{\frac{-\gamma}{1-\gamma}}(0)-\varpi \frac{\gamma}{(1-\gamma)} t}.
$$
Consequently, $B(t)$ blows-up in finite time
$$
T \leq T^{*}=\frac{1-\gamma}{\varpi \gamma B^{\frac{\gamma}{1-\gamma}(0)}}.
$$
This completes the proof.
\end{proof}
\section{Numerical Experiments}
\label{sec7}
\subsection{Preliminaries}
In this section we present numerical experiments illustrating the theoretical
results established in the previous sections. In particular, the simulations are
designed to exhibit the two main dynamical regimes predicted by the analysis:
exponential decay of solutions under suitable assumptions on the initial data,
and finite-time blow-up when the initial energy is negative.

To approximate the model, we employ an SBP--SAT finite-difference
discretization in space combined with a Newmark time-integration scheme.
This choice is motivated by the need to preserve the structural properties
of the continuous energy framework while providing a stable discretization
of the fourth-order operator.

The spatial domain is discretized as follows. The interval $[0,\pi]$ in the
$x$-direction is divided into $J$ uniform subintervals of length $\Delta x$,
while the interval $[-d,d]$ in the $y$-direction is divided into $K$
subintervals of length $\Delta y$. Hence, the domain $\Omega$ is partitioned
into rectangles of area $\Delta x\,\Delta y$. We write
\[
x_j=j\Delta x,\qquad j=0,1,\dots,J,
\]
and
\[
y_k=-d+k\Delta y,\qquad k=0,1,\dots,K.
\]
For the time variable, we integrate from $t=0$ up to a final time
$t=T$ using $N$ time steps of size $\Delta t=T/N$. The numerical
solution is denoted by $U_{j,k}^n\approx u(x_j,y_k,t_n)$. Accordingly, we write
\[
\boldsymbol{U}^n \in \mathbb{R}^{(J+1)(K+1)},\qquad
\boldsymbol{U}^n=
[\boldsymbol{U}_0^n\; \boldsymbol{U}_1^n\; \dots\; \boldsymbol{U}_{K}^n]^\top,
\]
where, for each fixed $k$,
\begin{equation}\label{vecU}
\boldsymbol{U}_k^n=
[U_{0,k}^n\; U_{1,k}^n\; \dots\; U_{J,k}^n]
\in\mathbb{R}^{J+1}.
\end{equation}
Thus, $\boldsymbol{U}_k^n$ contains the nodal values along the $x$-direction
at the fixed $y$-level $y_k$. In what follows, $\boldsymbol{U}^n$ is treated
as a column vector.
\subsection{Integration over time.}
The Newmark method \cite{krenk} is designed for problems of the form
\begin{equation}\label{prob_lineal}
\boldsymbol{M}\boldsymbol{\ddot{u}}(t) + \boldsymbol{C}\boldsymbol{\dot{u}}(t)+\boldsymbol{K}\boldsymbol{u}(t)=\boldsymbol{f}(t),
\end{equation}
where $\boldsymbol{M},\: \boldsymbol{C}$ and $\boldsymbol{K}$ are the mass, damping and stiffness matrices respectively, $\boldsymbol{f}(t)$ is an external load vector, and $\boldsymbol{u}(t)$ is the displacement. Writing $\boldsymbol{u}^n = \boldsymbol{u}(t_n)$ and so on for other vectors, the Newmark algorithm states
\begin{equation}\label{newmark}
\begin{cases}
&\left(\boldsymbol{M} + \gamma \Delta t \boldsymbol{C} + \beta\Delta t^2 \boldsymbol{K} \right)\boldsymbol{\ddot{u}}^{n+1} = \boldsymbol{f}^{n+1} - \boldsymbol{C}\left(\boldsymbol{\dot{u}}^n + (1-\gamma)\Delta t \boldsymbol{\ddot{u}}^n \right) \\ & \qquad \qquad \qquad \qquad \qquad \qquad \qquad  - \boldsymbol{K}\left(\boldsymbol{u}^n + \Delta t \boldsymbol{\dot{u}}^n + \left(\dfrac12 - \beta\right)\Delta t^2 \boldsymbol{\ddot{u}}^n \right), \\
&\boldsymbol{\dot{u}}^{n+1} = \boldsymbol{\dot{u}}^n + (1-\gamma)\Delta t  \boldsymbol{\ddot{u}}^n + \gamma \Delta t \boldsymbol{\ddot{u}}^{n+1}, \\
&\boldsymbol{u}^{n+1} = \boldsymbol{u}^n+\Delta t \boldsymbol{\dot{u}}^n+ \left(\dfrac12 - \beta\right)\Delta t^2 \boldsymbol{\ddot{u}}^n + \beta\Delta t^2 \boldsymbol{\ddot{u}}^{n+1},
\end{cases}
\end{equation}
where $\beta$ and $\gamma$ are real parameters introducing a bias in the solution. As we need a conservative scheme, we will choose $\gamma = \dfrac12$ and $\beta = \dfrac14$. From \eqref{eq2.5}, it is natural to choose $\boldsymbol{M}$ as the identity matrix. The remaining matrices $\boldsymbol{C}$, $\boldsymbol{K}$ and vector $\boldsymbol{f}$ will be defined in the following subsections. As $\boldsymbol{C}$ contains both the fractional derivative and memory terms, we will decompose it as $\boldsymbol{C} = \boldsymbol{C_1} + \boldsymbol{C_2}$, for $\boldsymbol{C_1}$ the fractional derivative effect and $\boldsymbol{C_2}$ the infinite memory one.

\subsection{Discretization of the bilaplacian.}
Let us focus on the discretization of the bilaplacian. This will eventually lead to the definition of $\boldsymbol{K}$ in \eqref{newmark}. We will use a second-order Summation-By-Parts (SBP) finite-difference operator in the $y$--direction together with Simultaneous
Approximation Terms (SAT) to weakly impose the boundary conditions at $y=\pm d$.  We assemble the discrete operator starting from the continuous bilinear form
associated with the plate energy (see the definition of $H_*^2$ and the inner
product in the preliminaries). For sufficiently smooth $u,v$,
\[
a(u,v)=\int_{\Omega}\Big(
u_{xx}v_{xx}+u_{yy}v_{yy}
+\sigma(u_{xx}v_{yy}+u_{yy}v_{xx})
+2(1-\sigma)u_{xy}v_{xy}
\Big)\,dx\,dy.
\]
This formulation is intrinsically two-dimensional and produces the natural
boundary conditions at $y=\pm d$ upon integration by parts; in particular, the
terms involving $2(1-\sigma)u_{xy}v_{xy}$ are essential to recover the correct
boundary fluxes. This construction preserves (in the discrete sense) the self-adjoint/energy structure of the continuous operator and avoids the non-symmetric boundary closures produced by ghost-node elimination. SBP--SAT methods and their energy stability properties are described in \cite{Fernandez2014review,Carpenter1999}, and SBP operators for second derivatives are developed in \cite{Mattsson2004}. For plate/biharmonic-type operators with free/clamped boundaries using SBP--SAT, see e.g. \cite{Tazhimbetov2023}. \\

\medskip
\noindent We impose the essential boundary conditions $u(0,y,t)=u(\pi,y,t)=0$ and work with the interior $x$-nodes $j=1,\dots,J-1$, so that $\boldsymbol U^n \in \mathbb{R}^{(J-1)(K+1)}$. The global vector is ordered by $y$-layers, i.e.
\[
\boldsymbol U^n
=
\big[
\boldsymbol U^n_{(\cdot,0)}\; ;\;
\boldsymbol U^n_{(\cdot,1)}\; ;\;\dots;\;
\boldsymbol U^n_{(\cdot,K)}
\big],
\quad
\boldsymbol U^n_{(\cdot,k)}
=
\big(U^n_{1,k},\dots,U^n_{J-1,k}\big)^\top\in\mathbb{R}^{J-1}.
\]
We will denote by $\boldsymbol{I}_x\in \mathbb{R}^{(J-1)\times (J-1)}$ and $\boldsymbol{I}_y\in \mathbb{R}^{(K+1)\times (K+1)}$ as identity matrices in their respective dimensions. Observe that Kronecker products of the form $\boldsymbol I_y\otimes(\cdot)$ act on
the $x$-direction within each $y$-layer.

\medskip
\noindent On the derivatives for the $x$ variable, we recall the centered first-derivative matrix
\[
\boldsymbol D_{1,x}
=
\frac{1}{\Delta x}
\begin{pmatrix}
-1 &  1 &        &        &        &  \\
-\tfrac12 & 0 & \tfrac12 &        &        &  \\
        & -\tfrac12 & 0 & \tfrac12 &        &  \\
        &        & \ddots & \ddots & \ddots &  \\
        &        &        & -\tfrac12 & 0 & \tfrac12 \\
        &        &        &        & -1 & 1
\end{pmatrix} \in \mathbb{R}^{(J-1)\times(J-1)},
\]
and centered second-derivative matrix on the interior nodes
$j=1,\dots,J-1$:
\[
\boldsymbol D_{x,1D}^2
=
\frac{1}{\Delta x^2}
\begin{pmatrix}
-2 & 1  &        &        &  \\
 1 & -2 & 1      &        &  \\
   & \ddots & \ddots & \ddots & \\
   &        & 1      & -2     & 1\\
   &        &        & 1      & -2
\end{pmatrix}
\in\mathbb{R}^{(J-1)\times(J-1)}.
\]
The corresponding 2D operator is
\[
\boldsymbol D_{xx} := \boldsymbol I_y\otimes \boldsymbol D_{x,1D}^2
\in\mathbb{R}^{(J-1)(K+1)\times(J-1)(K+1)}.
\]
Defining the SBP norm (trapezoidal rule)
\[
\boldsymbol H_y
=
\Delta y\,\mathrm{diag}\!\left(\tfrac12,1,\dots,1,\tfrac12\right)
\in\mathbb{R}^{(K+1)\times(K+1)}.
\]
along with $\boldsymbol H_x=\Delta x\,\boldsymbol I_x$ and $\boldsymbol{H} := \boldsymbol H_y \otimes \boldsymbol H_x $, then the corresponding 2D operator for the fourth derivative is
\[
\boldsymbol D_{xxxx} := \boldsymbol D_{xx}^\top \boldsymbol H \boldsymbol D_{xx}.
\]
On the SBP operator in $y$, let $\boldsymbol e_0=(1,0,\dots,0)^\top$ and $\boldsymbol e_K=(0,\dots,0,1)^\top$ denote the canonical vectors extracting the boundary values at $y=-d$ and $y=d$. We also set $\boldsymbol E_0=\boldsymbol e_0\boldsymbol e_0^\top$ and
$\boldsymbol E_K=\boldsymbol e_K\boldsymbol e_K^\top$. We take the second-order SBP first-derivative matrix
$\boldsymbol D_{1,y}\in\mathbb{R}^{(K+1)\times(K+1)}$ defined by centered
differences in the interior and one-sided differences at the boundaries:
\[
\boldsymbol D_{1,y}
=
\frac{1}{\Delta y}
\begin{pmatrix}
-1 &  1 &        &        &        &  \\
-\tfrac12 & 0 & \tfrac12 &        &        &  \\
        & -\tfrac12 & 0 & \tfrac12 &        &  \\
        &        & \ddots & \ddots & \ddots &  \\
        &        &        & -\tfrac12 & 0 & \tfrac12 \\
        &        &        &        & -1 & 1
\end{pmatrix}.
\]
The pair $(\boldsymbol H_y,\boldsymbol D_{1,y})$ satisfies the SBP identity
\[
\boldsymbol H_y\boldsymbol D_{1,y}
+
(\boldsymbol H_y\boldsymbol D_{1,y})^\top
=
\boldsymbol E_K-\boldsymbol E_0,
\]
which is the discrete analogue of integration by parts
\cite{Fernandez2014review,Carpenter1999}. For the second derivative, we use a second-order accurate matrix $\boldsymbol D_{2,y}\in\mathbb{R}^{(K+1)\times(K+1)}$ with centered interior
stencil and second-order one-sided closures:
\[
\boldsymbol D_{2,y}
=
\frac{1}{\Delta y^2}
\begin{pmatrix}
 2 & -5 &  4 & -1 &        &        \\
 1 & -2 &  1 &    &        &        \\
   &  1 & -2 &  1 &        &        \\
   &    & \ddots & \ddots & \ddots &  \\
   &    &        &  1 & -2 &  1 \\
   &    &        & -1 &  4 & -5 & 2
\end{pmatrix}.
\]
(Other SBP second-derivative operators can be used; see \cite{Mattsson2004} for
families of SBP $D_{2,y}$ operators and higher-order closures.) We define the 2D operators
\[
\boldsymbol D_{yy} := \boldsymbol D_{2,y}\otimes \boldsymbol I_x,
\qquad
\boldsymbol D_{y} := \boldsymbol D_{1,y}\otimes \boldsymbol I_x,
\qquad \boldsymbol{D}_{xy}=\boldsymbol{D}_{1,y}\otimes \boldsymbol{D}_{1,x},
\qquad \boldsymbol{D}_{yyyy} := \boldsymbol{D}_{2,y}^\top \boldsymbol H \boldsymbol D_{2,y}.
\]
\medskip
\noindent
With this notation, we approximate $\Delta^2 u = u_{xxxx}+u_{yyyy}+2u_{xxyy}$
\begin{equation}\label{bilap}
\boldsymbol K
:= \boldsymbol{H}^{-1}\left(
\boldsymbol D_{xxxx} + \boldsymbol{D}_{yyyy}
+ \sigma\big(\boldsymbol{D}_{xx}^\top \boldsymbol{H} \boldsymbol{D}_{yy}+\boldsymbol{D}_{yy}^\top \boldsymbol{H} \boldsymbol{D}_{xx}\big)
+2(1-\sigma)\boldsymbol{D}_{xy}^\top \boldsymbol{H} \boldsymbol{D}_{xy}\right),
\end{equation}
which will be the matrix $\boldsymbol{K}$ to be considered in \eqref{newmark}. With this construction, the boundary conditions are enforced without ghost nodes, and the resulting spatial discretization is energy-consistent and self-adjoint in the discrete $H$-inner product.
\subsection{Treatment of the Fractional Derivative.}
For the fractional derivative and the infinite memory term, we will follow the ideas from \cite{aslam25} and \cite{ammari25}. From Lemma \ref{le2} and due to the parity of $\phi$ with respect to $\vartheta$, the fractional derivative can be rewritten as
\begin{equation}\label{der_frac}
\partial_t^{\alpha,\beta}u(t) = \dfrac{2\sin(\alpha \pi)}{\pi}\int\limits_{0}^{\infty}\phi(x,y,\vartheta,t)\xi(\vartheta) d\vartheta,
\end{equation}
where $\xi(\vartheta) = |\vartheta|^{\frac{2\alpha - 1}{2}},\: \vartheta \in \mathbb{R},\: \alpha \in (0,1)$, and $\phi(x,y,\vartheta,t)$ verifies the problem
\begin{equation}\label{pvi_phi}
\begin{cases}
    &\phi_t(x,y,\vartheta,t)+(\vartheta^2 + \beta)\phi (x,y,\vartheta,t) - u_t(x,y,t)\xi(\vartheta) = 0,\: \vartheta \in \mathbb{R},\: t>0,\: \beta > 0, \\
    & \phi (x,y,\vartheta,0) = 0.
\end{cases}
\end{equation}
Observe that $\phi(x,y,\vartheta,t)$ and $\phi(x,y,-\vartheta,t)$ verify the same equation, which explains the parity property. Our proposal aims to approximate the improper integral \eqref{der_frac} as a definite integral
\[
\dfrac{2\sin(\alpha \pi)}{\pi}\int\limits_{0}^{\infty}\phi(x,y,\vartheta,t)\xi(\vartheta) d\vartheta \approx \dfrac{2\sin(\alpha \pi)}{\pi}\int\limits_{0}^{+R}\phi(x,y,\vartheta,t)\xi(\vartheta) d\vartheta ,
\]
for $R>0$ sufficiently large. Let $L\in \mathbb{N}$, $\Delta \vartheta = \dfrac{R}{L}$, $\vartheta_\ell = \ell \Delta \vartheta,\: \ell = 0,1,\dots, L$, $\xi_\ell = \xi(\vartheta_\ell)$. Let $\boldsymbol{\phi}_\ell^n = \phi(x,y,\vartheta_\ell,t_n) \in \mathbb{R}^{(J-1)(K+1)}$. The fractional derivative will be approximated as 
\begin{equation}\label{der_frac_disc}
\partial_t^{\alpha,\beta}u(t_n) \approx \dfrac{2\sin(\alpha \pi)}{\pi}\sum\limits_{\ell = 0}^{L} \boldsymbol{\phi}_\ell^n \xi_\ell \Delta \vartheta,
\end{equation}
where $\boldsymbol{\phi}_\ell^n$ is the numerical solution of \eqref{pvi_phi}, and is computed by the following Crank-Nicolson scheme:
\begin{equation}\label{esq_phi}
\begin{cases}
\boldsymbol{\phi}_\ell^{n+1} = \dfrac{2-\Delta t(\vartheta^2+\beta)}{2+\Delta t(\vartheta^2+\beta)}\boldsymbol{\phi}_\ell^n + \dfrac{\Delta t}{2+\Delta t(\vartheta^2 + \beta)}\boldsymbol{\dot{U}}^{n+\frac12}\xi_\ell,\quad n=0,1,\dots,N,\quad \ell = 0,\dots,L; \\
\boldsymbol{\phi}_\ell^{0} = \boldsymbol{\Theta},\: \ell = 0,\dots,L;
\end{cases}
\end{equation}
where $\boldsymbol{\Theta} \in \mathbb{R}^{(J-1)(K+1)}$ is the null matrix, and $\boldsymbol{\dot{U}}^n = \frac{\boldsymbol{U}^{n+1} - \boldsymbol{U}^n}{\Delta t}$. In contrast to other variables, $\boldsymbol{\phi}^n = \{ \boldsymbol{\phi}_\ell^n\}_{\ell=0}^{\ell = L}$ is a three dimensional array which imposes a heavy use of computer memory. Observe \eqref{esq_phi} is an explicit scheme in time, and therefore there is no need to store all values of $\boldsymbol{\phi}^n$ in every timestep. With this, in \eqref{newmark} we will have
\begin{align}
\boldsymbol{C_1}\boldsymbol{\dot{U}}^{n+1} &= \dfrac{2\sin(\alpha \pi)}{\pi}\sum\limits_{\ell = 0}^{L} \boldsymbol{\phi}_\ell^{n+1} \xi_\ell \Delta \vartheta \nonumber \\
&= \dfrac{2\sin(\alpha \pi)}{\pi}\sum\limits_{\ell = 0}^{L} \left(\dfrac{2-\Delta t (\vartheta^2+\beta)}{2+\Delta t (\vartheta^2+\beta)}\boldsymbol{\phi}_\ell^n + \dfrac{2\xi(\vartheta_\ell)}{2+\Delta t (\vartheta^2+\beta)}\boldsymbol{\dot{U}}^{n+\frac12}\right) \xi_\ell \Delta \vartheta \nonumber \\
&= \sum\limits_{\ell=0}^L\kappa_\ell \boldsymbol{\phi}_\ell^n \Delta \vartheta + \sum\limits_{\ell=0}^L k_\ell \boldsymbol{\dot{U}}^{n+1} \Delta \vartheta + \sum\limits_{\ell=0}^L k_\ell \boldsymbol{\dot{U}}^{n} \Delta \vartheta, \nonumber
\end{align}
where
\[ \kappa_\ell := \dfrac{2a_1\sin(\alpha \pi)\xi(\vartheta_\ell) (2-\Delta t (\vartheta^2+\beta))}{\pi(2+\Delta t (\vartheta^2+\beta))},\qquad k_\ell := \dfrac{a_1\sin(\alpha\pi)\xi^2(\vartheta_\ell) \Delta t}{\pi (2+\Delta t(\vartheta^2+\beta))}. \]
And from \eqref{newmark}, we obtain
\begin{equation}\label{C1}
\boldsymbol{C_1}\boldsymbol{\dot{U}}^{n+1} = \sum\limits_{\ell=0}^L\kappa_\ell \boldsymbol{\phi}_\ell^n \Delta \vartheta + 2\sum\limits_{\ell=0}^L k_\ell \boldsymbol{\dot{U}}^{n} \Delta \vartheta + \dfrac{\Delta t}{2} \sum\limits_{\ell=0}^L k_\ell \boldsymbol{\ddot{U}}^{n+1} \Delta \vartheta + \dfrac{\Delta t }{2} \sum\limits_{\ell=0}^L k_\ell \boldsymbol{\ddot{U}}^{n} \Delta \vartheta.
\end{equation}
\subsection{Infinite memory term.}
From \eqref{eq12}, we have
\[
-\int\limits_0^\infty g(s) \Delta^2u(x,y,t-s)ds = \int\limits_0^\infty g(s)\Delta^2\mu(x,y,t,s)ds,
\]
where $\mu$ verifies \eqref{2.3m}. For $t=0$, observe that
\[ \mu(x,y,0,s) = u_0(x,y)-u(x,y,-s) = u_0(x,y) - h(x,y,s).\]
Similar to what we have done for the fractional derivative, we will approximate the improper integral as follows:
\[
\int\limits_0^\infty g(s)\Delta^2\mu(x,y,t,s)ds \approx \int\limits_0^S g(s)\Delta^2\mu(x,y,t,s)ds,
\]
for $S>0$ sufficiently big. Let $M\in\mathbb{N}$. We will write $s_m := m\Delta s,\: m=0,1,\dots,M$, $\Delta s = \frac{S}{M}$, and $\boldsymbol{\mu}_m^n = \mu(x,y,t_n,s_m) \in \mathbb{R}^{(J-1)(K+1)}$. The previous integral will be computed as follows
\begin{equation}\label{memo}
\int\limits_0^S g(s)\Delta^2\mu(x,y,t_n,s)ds \approx \sum\limits_{m=0}^M g(s_m) \boldsymbol{K} \boldsymbol{\mu}_m^n \Delta s,
\end{equation}
where $\boldsymbol{K}$ is defined in \eqref{bilap}, and $\boldsymbol{\mu}_m^n$ is the numerical solution of \eqref{2.3m} and is computed through the following upwind finite volume scheme: 
\begin{equation}\label{esq_mu}
\begin{cases}
    \boldsymbol{\mu}_{m}^{n+1} = \left(1-\dfrac{\Delta t}{\Delta s} \right)\boldsymbol{\mu}_{m}^{n} +\dfrac{\Delta t}{\Delta s}\boldsymbol{\mu}_{m-1}^{n} + \Delta t \boldsymbol{\dot{U}}^{n+\frac12},\: n=0,1,\dots,N \\
    \boldsymbol{\mu}_{m}^{0} = \boldsymbol{U}^0 - \boldsymbol{h}_m,
    \end{cases}
\end{equation}
for $\boldsymbol{h}_m = h(x,y,s_m) \in \mathbb{R}^{(J-1)(K+1)}$. The use of $\boldsymbol K$ in \eqref{memo} instead of another bilaplacian matrix is coherent with the definition \eqref{eq12}, as both $\mu(x,y,t,s)$ and $u(x,y,t)$ functions fulfill the same boundary conditions at $x =0$, $x=\pi$ and $y = \pm d$. Thus, at $t = t_{n+1}$ and due to the Newmark relations in \eqref{newmark}, the approximation \eqref{memo} turns into
\begin{align} 
\boldsymbol{C_2}\boldsymbol{\dot{U}}^{n+1}&= \sum\limits_{m=0}^M g(s_m) \boldsymbol{K} \boldsymbol{\mu}_m^{n+1} \Delta s \nonumber \\
&= \sum\limits_{m=0}^M g(s_m) \boldsymbol{K}\left( \left(1-\dfrac{\Delta t}{\Delta s} \right) \boldsymbol{\mu}_m^{n} + \dfrac{\Delta t}{\Delta s}\boldsymbol{\mu}_{m-1}^{n}  + \Delta t \boldsymbol{\dot{U}}^{n+\frac12}\right)\Delta s \nonumber \\
&= \left(1-\dfrac{\Delta t}{\Delta s} \right)\sum\limits_{m=0}^M g(s_m) \boldsymbol{K}  \boldsymbol{\mu}_m^{n} \Delta s
+ \sum\limits_{m=0}^M g(s_m)\boldsymbol{K\mu}_{m-1}^{n}\Delta t  \nonumber \\
&\qquad+ \sum\limits_{m=0}^M g(s_m)\dfrac{\boldsymbol{K\dot{U}}^{n+1}+ \boldsymbol{K\dot{U}}^{n}}{2} \Delta s \nonumber \\
&= \left(1-\dfrac{\Delta t}{\Delta s} \right)\sum\limits_{m=0}^M g(s_m) \boldsymbol{K}  \boldsymbol{\mu}_m^{n} \Delta s + \sum\limits_{m=0}^M g(s_m)\boldsymbol{K\mu}_{m-1}^{n}\Delta t  \nonumber \\
&+ \dfrac{\Delta t}{4} \sum\limits_{m=0}^M g(s_m) \boldsymbol{K}\left(\boldsymbol{\ddot{U}}^{n+1}+\boldsymbol{\ddot{U}}^n\right) \Delta s + \sum\limits_{m=0}^M g(s_m) \boldsymbol{K \dot{U}}^n \Delta s \label{C2}
\end{align}
\subsection{Nonlinear term}
Regarding the right hand side of \eqref{eq2.5}$_1$, we will follow the ideas from Delfour, Fortin and Payre \cite{delfour} and define the following operator
\begin{equation}\label{fun_J}
\mathcal{J}(a,b):= \begin{cases} \dfrac1{p}\dfrac{|a|^{p}-|b|^{p} }{(|a|^2-|b|^2)}(a+b),\quad &\text{ if } a\neq b, \\ |a|^{p-2}a & \text{ if } a = b. \end{cases}
\end{equation}
This is done in order to ensure the energy conservation if no damping terms are present. 
\subsection{Assembly of the numerical scheme}
In order to assembly the scheme, we will start from \eqref{prob_lineal} at $t = t_{n+1}$, incorporating \eqref{fun_J}, writing $\boldsymbol{C} = \boldsymbol{C_1} + \boldsymbol{C_2}$ and considering $\boldsymbol{M}$ as the identity matrix $\boldsymbol I$. Thus, we have 
\begin{equation}\label{esq_num_1}
    \boldsymbol{\ddot{U}}^{n+1} + (\boldsymbol{C_1+C_2})\boldsymbol{\dot{U}}^{n+1}+\boldsymbol{KU}^{n+1} = \mathcal{J}(\boldsymbol{U}^{n+1},\boldsymbol{U}^n).
\end{equation}
As $\boldsymbol{K}$ was defined in \eqref{bilap}, we only need to incorporate $\boldsymbol{C_1}$ and $\boldsymbol{C_2}$ from \eqref{C1} and \eqref{C2}, respectively. Thus, \eqref{esq_num_1} combined with \eqref{esq_phi} and \eqref{esq_mu} turns into the numerical scheme we will use to replicate the results proved in previous sections:
\begin{equation}\label{esq_num}
\begin{cases}
        &\displaystyle \left( \boldsymbol{I} + \dfrac{\Delta t}{2} \sum\limits_{\ell=0}^L k_\ell \boldsymbol{I} + \dfrac{\Delta t}{4} \sum\limits_{m=0}^M g(s_m) \boldsymbol{K} \Delta s + \dfrac{\Delta t^2}{4}\boldsymbol{K} \right)\boldsymbol{\ddot{U}}^{n+1}  \\
        &\displaystyle \qquad = \mathcal{J}(\boldsymbol{U}^{n+1},\boldsymbol{U}^n) - \boldsymbol{K}\left(\boldsymbol{U}^n + \Delta t \boldsymbol{\dot{U}}^n + \dfrac{\Delta t^2}{4}\boldsymbol{\ddot{U}}^n \right) - \sum\limits_{\ell=0}^L\kappa_\ell \boldsymbol{\phi}_\ell^n \Delta \vartheta \\
        &\displaystyle \qquad - 2\sum\limits_{\ell=0}^L k_\ell \boldsymbol{\dot{U}}^{n} \Delta \vartheta \Delta \vartheta - \dfrac{\Delta t }{2} \sum\limits_{\ell=0}^L k_\ell \boldsymbol{\ddot{U}}^{n} \Delta \vartheta 
        - \left(1-\dfrac{\Delta t}{\Delta s} \right)\sum\limits_{m=0}^M g(s_m) \boldsymbol{K}  \boldsymbol{\mu}_m^{n} \Delta s \\
        & \displaystyle \qquad - \sum\limits_{m=0}^M g(s_m)\boldsymbol{K\mu}_{m-1}^{n}\Delta t  - \dfrac{\Delta t}{4} \sum\limits_{m=0}^M g(s_m) \boldsymbol{K}\boldsymbol{\ddot{U}}^n\Delta s + \sum\limits_{m=0}^M g(s_m) \boldsymbol{K \dot{U}}^n \Delta s \\
        & \boldsymbol{\dot{U}}^{n+1} = \boldsymbol{\dot{U}}^{n}+ \dfrac{\Delta t}{2}\boldsymbol{\ddot U}^{n} + \dfrac{\Delta t}{2}\boldsymbol{\ddot U}^{n+1} \\
        & \boldsymbol{U}^{n+1} = \boldsymbol{U}^{n} + \Delta t \boldsymbol{\dot{U}}^{n} + \dfrac{\Delta t^2}{4}\boldsymbol{\ddot{U}}^{n} + \dfrac{\Delta t^2}{4}\boldsymbol{\ddot{U}}^{n+1} \\
        & \boldsymbol{U}_{j,k}^0 = u_0(x_j,y_k),\: j=0,1,\dots,J, k = 0,1,\dots,K; \\
        &\boldsymbol{\mu}_{m}^{n+1} = \left(1-\dfrac{\Delta t}{\Delta s} \right)\boldsymbol{\mu}_{m}^{n} +\dfrac{\Delta t}{\Delta s}\boldsymbol{\mu}_{m-1}^{n} + \Delta t \boldsymbol{\dot{U}}^{n+\frac12}, \quad n=0,1,\dots,N,\: m=0,1,\dots,M; \\
        &\boldsymbol{\mu}_{m}^{0} = \boldsymbol{U}^0 - \boldsymbol{h}_m,\: m=0,1,\dots,M; \\
        &\boldsymbol{\phi}_\ell^{n+1} = \dfrac{2-\Delta t(\vartheta^2+\beta)}{2+\Delta t(\vartheta^2+\beta)}\boldsymbol{\phi}_\ell^n + \dfrac{\Delta t}{2+\Delta t(\vartheta^2 + \beta)}\boldsymbol{\dot{U}}^{n+\frac12}\xi_\ell,\: \ell = 0,\dots,L,\: n=0,1,\dots,N; \\
&\boldsymbol{\phi}_\ell^{0} = \boldsymbol{\Theta},\: \ell = 0,\dots,L.
\end{cases}
\end{equation}
Because \eqref{esq_num}$_1$ is a nonlinear problem, we will solve it using a fixed-point iteration. Therefore, we will solve a linear system of equations in each timestep. Also in each timestep we will compute the discrete version of the energy defined in \eqref{eq2.6}, which will be given by
\begin{equation}\label{energy_num}
\begin{aligned}
E^n &:= \dfrac{\lambda}{2} (\boldsymbol{U}^n)^\top \boldsymbol{K} \boldsymbol{U}^n \Delta x \Delta y + \dfrac{1}{2}(\boldsymbol{\dot U}^n)^\top \boldsymbol{\dot U}^n\Delta x \Delta y \\
&\qquad + \dfrac{\kappa}{2} \sum_{\ell=0}^L (\boldsymbol{\phi}_\ell^n)^\top \boldsymbol{\phi}_\ell^n\Delta\vartheta \Delta x \Delta y  -\dfrac1p  \left((\boldsymbol{U}^n)^\top \boldsymbol{U}^n\right)^{\frac{p}{2}} \Delta x \Delta y \\
&\qquad + \dfrac12 \sum\limits_{m=0}^M g(s_m) \Big( \left(\boldsymbol{D}_{xx}\boldsymbol\mu_m^n\right)^\top \left(\boldsymbol{D}_{xx}\boldsymbol\mu_m^n\right) + \left(\boldsymbol{D}_{yy}\boldsymbol\mu_m^n\right)^\top \left(\boldsymbol{D}_{yy}\boldsymbol\mu_m^n\right) \\
&\qquad + 2\sigma \left(\boldsymbol{D}_{xx}\boldsymbol\mu_m^n\right)^\top \left(\boldsymbol{D}_{yy}\boldsymbol\mu_m^n\right) + 2(1-\sigma)\left(\boldsymbol{D}_{xy}\boldsymbol\mu_m^n\right)^\top \left(\boldsymbol{D}_{xy}\boldsymbol\mu_m^n\right)   \Big)\Delta s
\end{aligned}
\end{equation}
The following experiments illustrate the theoretical results on stability and
blow-up derived in Sections \ref{sec_stability} and \ref{sec_blowup}.
\subsection{Computational results: exponential decay}

\begin{figure}[hbt]
    \centering
    \includegraphics[width=0.8\linewidth]{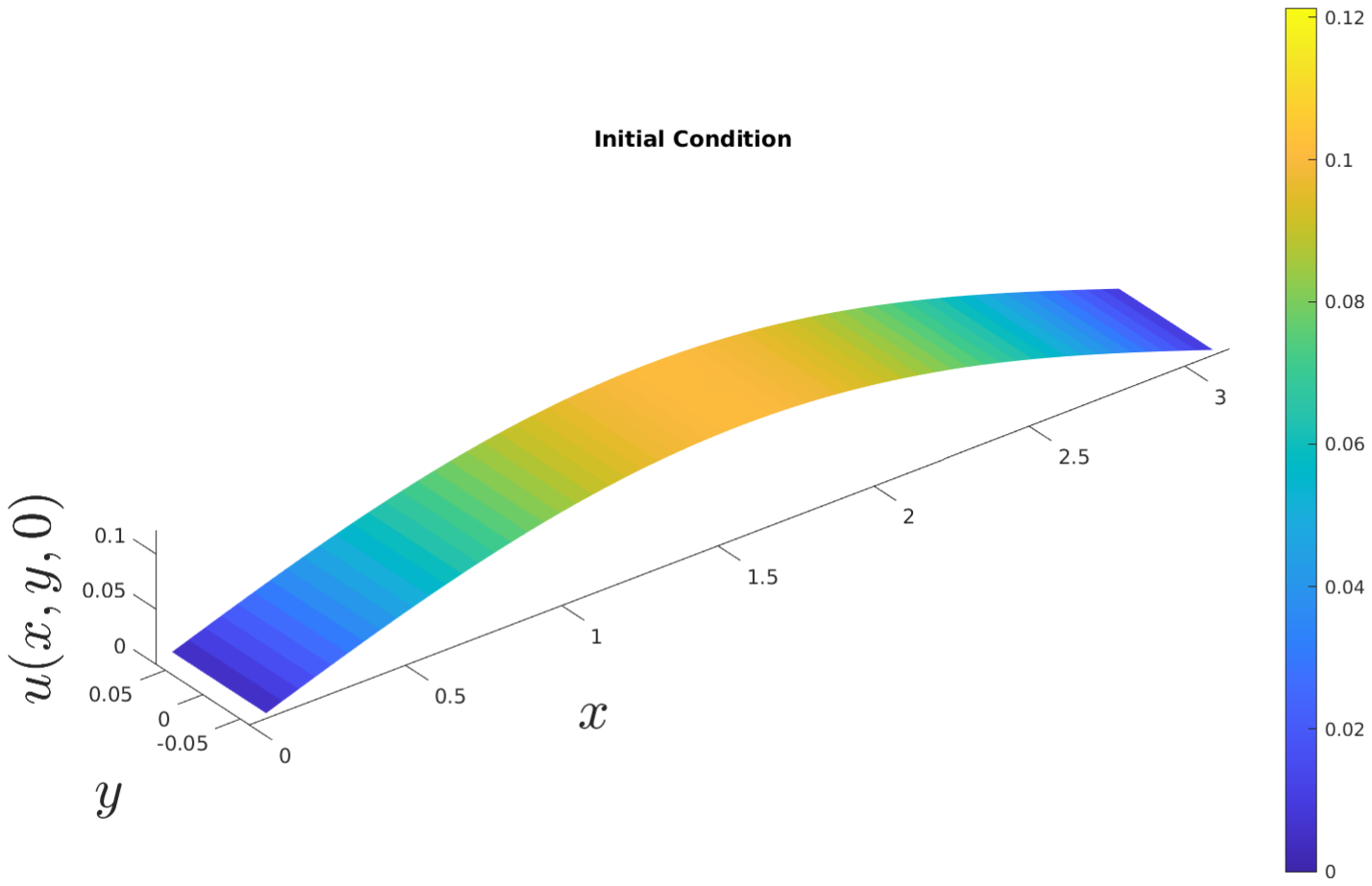}
    \caption{Initial displacement profile $u(x,y,0)$ of the bridge deck
    obtained from the truncated solution of the static problem
    $\Delta^2 u(x,y)=\sin(x)/10$.}
    \label{fig_initial}
\end{figure}

In this subsection we illustrate the exponential stability result stated in
Theorem~\ref{thm5.5} for several values of $p>2$. In all simulations we take
\[
d=\frac{\pi}{50},\qquad R=10\pi,\qquad L=100,\qquad T=1000,\qquad S=80,
\]
\[
\alpha=0.95,\qquad \beta=2.5,\qquad \sigma=0.1,\qquad \lambda=\frac12,\qquad a_1=1,
\]
together with
$
\Delta x=\dfrac{\pi}{60}\approx 0.05236,\
\Delta y=\dfrac{\pi}{500}\approx 0.00628,\
\Delta t=0.001,\
\Delta s=\dfrac12,$ and 
$
\Delta\vartheta=\dfrac{2R}{L}\approx 0.62831.
$

The initial data are chosen as
\[
g(s)=\frac{e^{-2s}}{10000},\qquad h(t)=e^{2t},\qquad u_t(x,y,0)=0,
\]
while $u(x,y,0)$ is taken as the truncation to the first $100$ terms of the
exact solution of the static problem (see \cite{Ferr1}, Theorem~3.2)
\[
\Delta^2u(x,y)=\frac{\sin(x)}{10}.
\]

Figure~\ref{fig_initial} illustrates the initial displacement profile
$u(x,y,0)$ of the bridge deck used in the simulations. This profile
corresponds to the truncated series representation of the solution of
the static problem described above.

In order to compare the numerical experiments with the sufficient condition
given in Lemma~\ref{lem4.1}, we estimate the embedding constant $C_e$
using the discrete procedure described in Appendix~\ref{app1}, based on
the first generalized eigenpair of the stiffness matrix associated with
the SBP discretization.

Table~\ref{table1} reports the resulting numerical estimates of $C_e$,
together with the left- and right-hand sides of the first inequality
in \eqref{eq4.3}, for the values of $p$ considered in the simulations.
All computations in Table~\ref{table1} were performed with
$J=60$, $K=20$, $d=\pi/50$, $\sigma=0.1$, and $\lambda=0.5$.
The third column corresponds to the left-hand side of the condition in
Lemma~\ref{lem4.1}, whereas the fourth column gives the right-hand side.

We observe that the condition of Lemma~\ref{lem4.1} is satisfied for
$p=3,4,5$, but not for $p=2.5$. Therefore, the theoretical hypothesis
ensuring exponential decay is verified only in the first three cases.
Nevertheless, as shown in Figure~\ref{fig1}, the numerical solution also
exhibits decay for $p=2.5$. This indicates that the condition in
Lemma~\ref{lem4.1} is sufficient but not necessary, and suggests that the
theoretical threshold may be conservative.

\begin{table}[hbt]
    \centering
    \begin{tabular}{c|c|c|c}
         $p$ & $C_e$& $C_{e}\left(\frac{2 p}{p-2} E(0)\right)^{\frac{p-2}{2}}$ & $\lambda^{\frac{p}{2}}$  \\ \hline
         $2.5$ & {\tt 3.351865562758652}& {\tt 0.792476072196316} & {\tt 0.420448207626857}   \\ \hline
         $3$ & {\tt 5.632082885891357} & {\tt 0.305108317014464} & {\tt 0.353553390593274}  \\ \hline
         $4$ & {\tt 16.337171741287030} & {\tt 0.035470019976350} & {\tt 0.25} \\ \hline
         $5$ & {\tt 48.561258580598867} & {\tt 0.003774083392963} & {\tt 0.176776695296637} \\
    \end{tabular}
    \caption{Estimated values of $C_e$ and of the two sides of the first inequality in \eqref{eq4.3}. The computations were carried out with $J=60$, $K=20$, $d=\pi/50$, $\sigma=0.1$, and $\lambda=0.5$.}
    \label{table1}
\end{table}

\begin{figure}[h]
    \centering
    \includegraphics[width=0.8\linewidth]{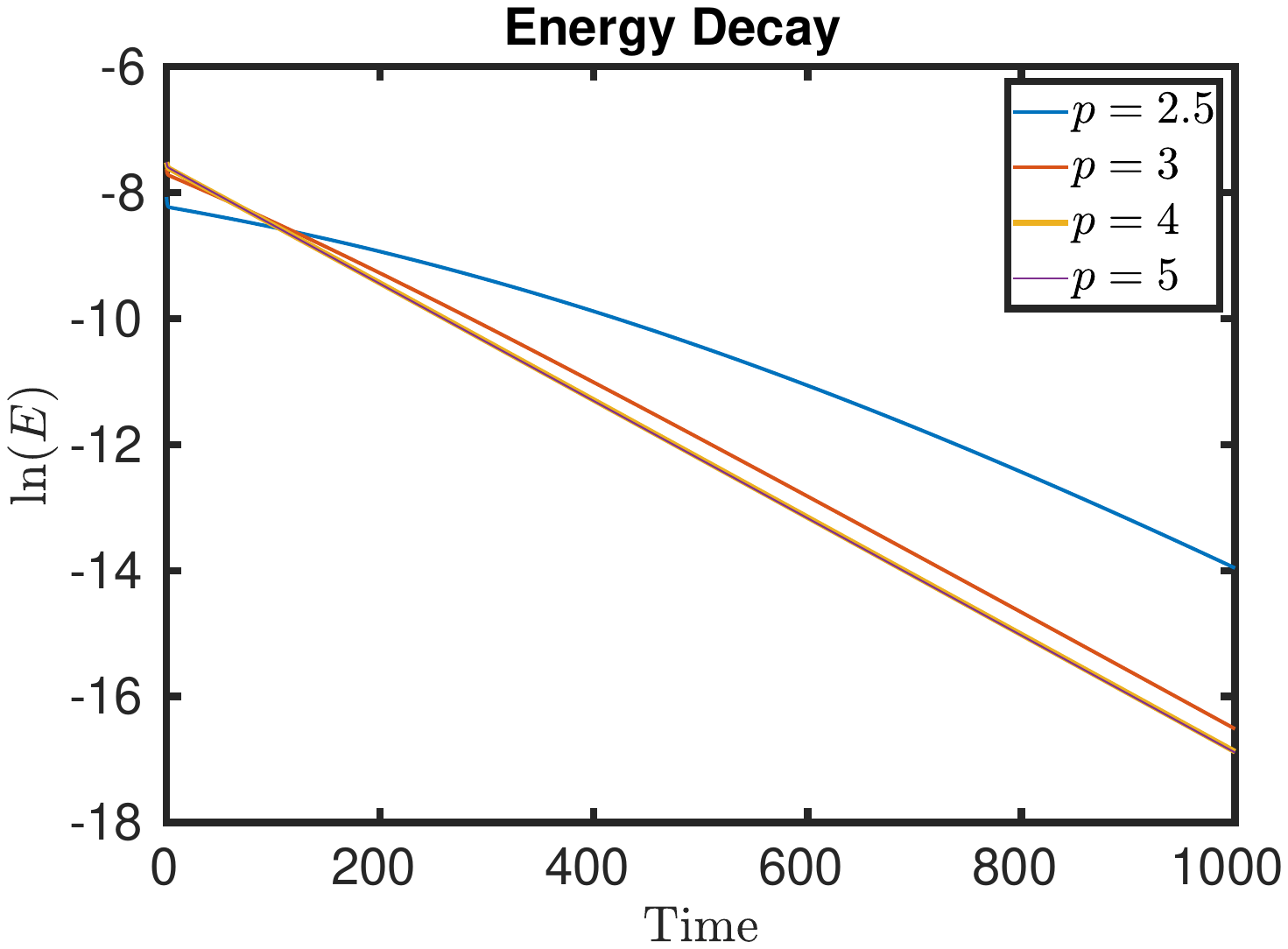}
    \caption{Evolution of the discrete energy for different values of $p$, illustrating the decay regime.}
    \label{fig1}
\end{figure}

Figure~\ref{fig1} displays the evolution of the discrete energy $E^n$
defined in \eqref{energy_num} for the above values of $p$.
For $p=3,4,5$, the observed decay is fully consistent with the theoretical
prediction of Theorem~\ref{thm5.5}, since in these cases the assumptions of
Lemma~\ref{lem4.1} are satisfied. For $p=2.5$, although the sufficient
condition is not fulfilled, the numerical solution still exhibits a decaying
regime, with a much longer transient phase before the exponential behavior
becomes clearly visible. This suggests that the analytical condition may be
relaxed, which could be an interesting direction for future investigation.
\subsection{Computational results: blow-up in finite time}
We conclude with a numerical experiment illustrating the blow-up
phenomenon predicted by Theorem~\ref{teo6.2}. In this case we choose
\[
p=2.1,\qquad d=\frac{\pi}{50},\qquad R=10\pi,\qquad L=100,\qquad
T=1800,\qquad S=80,
\]
\[
\alpha=\frac12,\qquad \beta=3.5,\qquad
\sigma=0.1,\qquad \lambda=\frac45,\qquad a_1=1,
\]
together with
$\Delta x=\dfrac{\pi}{60}\approx 0.05236,\
\Delta y=\dfrac{\pi}{500}\approx 0.00628,\
\Delta t=0.001,\
\Delta s=\dfrac12$ and $
\Delta\vartheta=\dfrac{2R}{L}\approx 0.62831.
$

The initial data are chosen as
\[
g(s)=\frac{e^{-5s}}{1000},\qquad
h(t)=e^{5t},\qquad
u_t(x,y,0)=0.
\]
In contrast with the previous subsection, the initial displacement
$u(x,y,0)$ is obtained from the solution of the static problem

\[
\Delta^2 u(x,y)=5\sin(x),
\]

which produces a significantly larger deformation of the bridge deck.
Figure~\ref{fig_initial_blowup} illustrates the resulting initial
configuration used in the simulation.

\begin{figure}[h]
    \centering
    \includegraphics[width=0.8\linewidth]{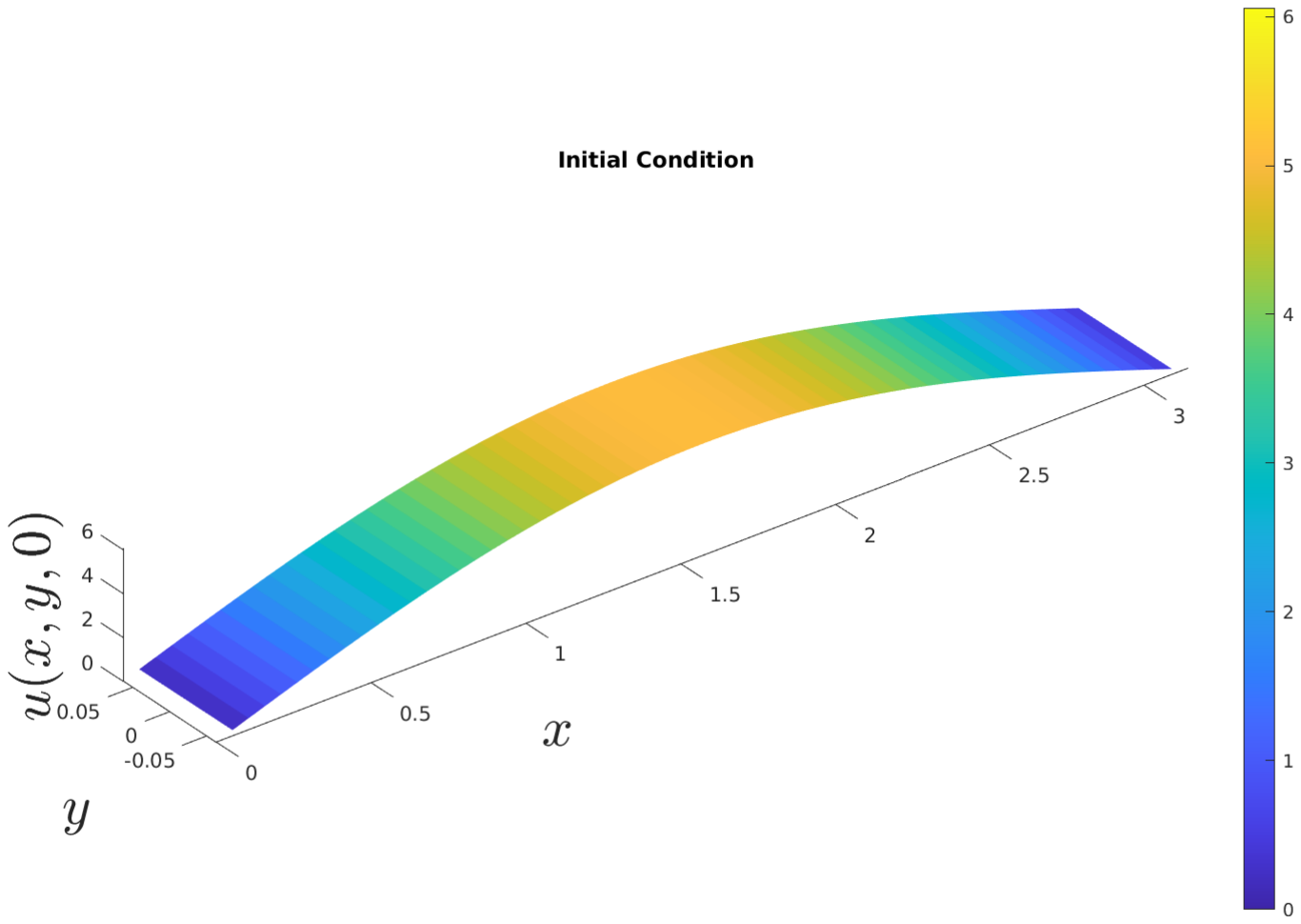}
    \caption{Initial displacement profile $u(x,y,0)$ used in the blow-up
    experiment. The profile is obtained from the solution of the static
    problem $\Delta^2 u(x,y)=5\sin(x)$.}
    \label{fig_initial_blowup}
\end{figure}

For this configuration the discrete initial energy satisfies
\[
E(0)=-0.483197525092926<0,
\]
with $p=2.1$, and
which verifies the hypothesis of Theorem~\ref{teo6.2}.

Figure~\ref{fig2} shows the time evolution of the numerical solution.
A rapid growth of the amplitude is clearly observed, which is consistent
with the finite-time blow-up behavior predicted by the theoretical
analysis when the initial energy is negative.

\begin{figure}[hbt]
    \centering
    \includegraphics[width=0.8\linewidth]{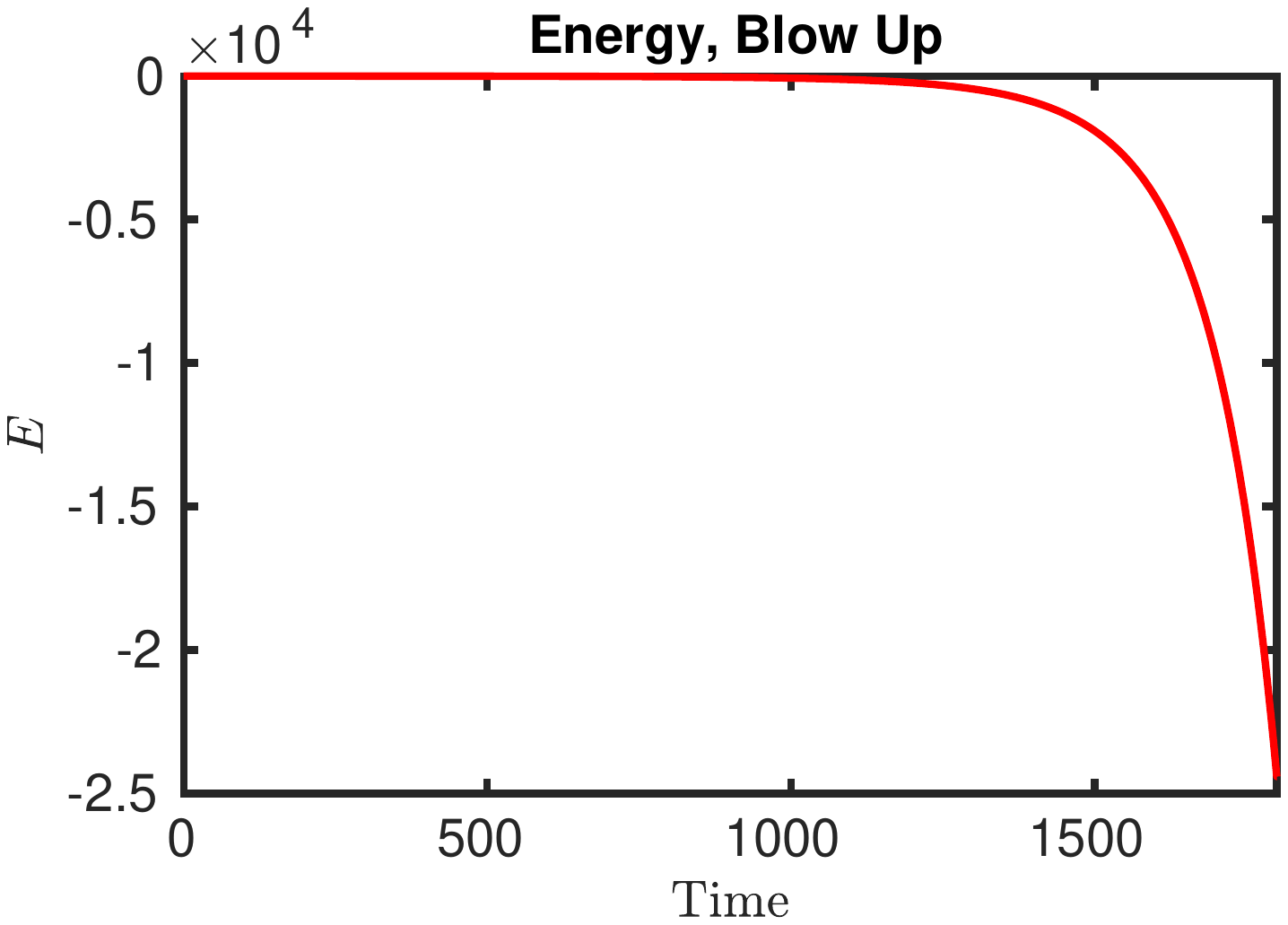}
    \caption{Time evolution of the numerical solution in the blow-up regime.}
    \label{fig2}
\end{figure}
\section{Conclusion and Future Work}
In this paper we analyzed a suspension bridge model governed by a nonlinear
plate equation that incorporates fractional damping, an infinite-memory term,
and a nonlinear source. Using semigroup techniques, we first established the
local well-posedness of the system in an appropriate energy space. We then
proved global existence and exponential stability of solutions under suitable
conditions on the initial data. In contrast, when the initial energy is
negative, we showed that the corresponding solutions blow up in finite time.
These results reveal the presence of a threshold phenomenon in the dynamics of
the model and highlight the interplay between dissipative effects and nonlinear
instability.

To complement the theoretical analysis, we also proposed a numerical
approximation based on SBP--SAT finite differences combined with a Newmark
time-integration scheme. The numerical experiments illustrate the stability and
blow-up regimes predicted by the analytical results. In particular, the
simulations show how different initial configurations of the bridge deck
lead to qualitatively distinct dynamical behaviors. The experiments include
two representative initial profiles obtained from stationary biharmonic
problems, which are displayed in Figures~\ref{fig_initial} and
\ref{fig_initial_blowup}. These configurations illustrate how the magnitude
of the initial deformation influences the sign of the initial energy and,
consequently, the long-term evolution of the system. In the case where the
initial energy is positive, the solution exhibits a decaying regime, whereas
a negative initial energy leads to the blow-up behavior predicted by the
theoretical analysis.

Several extensions of this work are possible. In particular, it would be
interesting to study related suspension bridge models involving alternative
damping mechanisms, such as structural damping of the form $-\Delta v_t$,
strong damping of the form $\Delta^2 v_t$, or viscoelastic effects described by
memory terms of the type
\[
-\int_0^t g(t-s)\Delta^2 v(s)\,ds,
\]
where $g$ denotes a relaxation kernel. Another promising direction is the
analysis of models combining such mechanisms with fractional time-delay
effects, as well as the investigation of the long-time dynamics and asymptotic
behavior of these systems.

An additional observation arising from the numerical experiments concerns
the threshold condition in Lemma~\ref{lem4.1}. While this condition is
satisfied for several values of $p$ considered in the simulations, the case
$p=2.5$ still exhibits a decaying regime even though the corresponding
inequality is not fulfilled. This suggests that the analytical condition is
sufficient but not necessary, and indicates that a sharper theoretical
threshold for exponential decay may be possible. Investigating such an
improvement remains an interesting topic for future work.

{\bf Acknowledgments:} Zayd Hajjej is supported by Ongoing Research Funding Program (ORF-2026-736), King Saud University, Riyadh, Saudi Arabia.
Mauricio Sepúlveda-Cortés is supported by Fondecyt-ANID project 1220869,
and
ANID-Chile through Centro de Modelamiento  Matem\'atico  (FB210005).

\appendix
\section{Appendix: Estimation of the embedding constant}\label{app1}

The constant $C_e$ appearing in Lemma \ref{lem2.2} arises from the Sobolev
embedding $H_*^2(\Omega_d)\hookrightarrow L^r(\Omega_d)$ and is therefore
not explicitly known. However, in the numerical experiments we require
an estimate of this constant in order to verify the condition in
Lemma \ref{lem4.1} ensuring exponential decay.

In this appendix we briefly explain how such an estimate can be obtained.
First, we derive a theoretical bound showing how the embedding constant
depends on the strip width $d$. We then describe a practical numerical
approximation based on the matrices arising from the SBP discretization.

\subsection*{Continuous estimate}

Let $\Omega_d=(0,\pi)\times(-d,d)$ with $0<d\le1$, and let
\[
H_*^2(\Omega_d)=\{u\in H^2(\Omega_d): u=0 \text{ on } \{0,\pi\}\times(-d,d)\}.
\]

\begin{lemma}
Let $1\le r<\infty$. Then there exists a constant $C_{r,\sigma}>0$,
independent of $d$, such that
\[
\|u\|_{L^r(\Omega_d)}^r
\le C_{r,\sigma}\, d^{\,1-\frac r2}\,
\|u\|_{H_*^2(\Omega_d)}^r,
\qquad \forall u\in H_*^2(\Omega_d).
\]
Consequently, the embedding constant in Lemma \ref{lem2.2} satisfies
\[
C_e(d,r)\le C_{r,\sigma}\, d^{\,1-\frac r2}.
\]
\end{lemma}

\begin{proof}
Let $u\in H_*^2(\Omega_d)$ and define $v(x,\eta)=u(x,d\eta)$ on
$\Omega_1=(0,\pi)\times(-1,1)$. A change of variables gives
\[
\|u\|_{L^r(\Omega_d)}^r
= d\,\|v\|_{L^r(\Omega_1)}^r.
\]
Since $\Omega_1$ is fixed, the Sobolev embedding
$H^2(\Omega_1)\hookrightarrow L^r(\Omega_1)$ yields
\[
\|v\|_{L^r(\Omega_1)}^r \le C_r \|v\|_{H^2(\Omega_1)}^r.
\]
Differentiation shows that
\[
\|v\|_{H^2(\Omega_1)}^r \le d^{-r/2}\|u\|_{H^2(\Omega_d)}^r,
\]
which gives
\[
\|u\|_{L^r(\Omega_d)}^r
\le C_r\, d^{1-r/2}\,\|u\|_{H^2(\Omega_d)}^r.
\]
Finally, the equivalence between the $H^2$ and $H_*^2$ norms
yields the result.
\end{proof}

\subsection*{Discrete approximation}

For the numerical simulations we approximate the embedding constant
using the discrete matrices arising from the SBP discretization.
Let $K$ denote the stiffness matrix and $W$ the diagonal quadrature
matrix associated with the SBP norm. For a discrete vector
$U\in\mathbb{R}^N$, define
\[
\|U\|_{p,h}^p=\sum_{i=1}^N w_i |U_i|^p,
\qquad
\|U\|_{*,h}^2=U^\top K U .
\]

A discrete analogue of the embedding constant is
\[
C_{e,h}(p)=
\sup_{U\neq0}
\frac{\|U\|_{p,h}^p}{(U^\top K U)^{p/2}}.
\]

When $p=2$ this reduces to the generalized Rayleigh quotient
\[
C_{e,h}(2)
=
\sup_{U\neq0}\frac{U^\top W U}{U^\top K U}
=
\frac{1}{\lambda_1},
\]
where $\lambda_1$ is the smallest generalized eigenvalue of
\[
K\phi=\lambda W\phi .
\]

The eigenpair $(\lambda_1,\phi_1)$ can be computed efficiently
using standard numerical methods for symmetric generalized
eigenvalue problems (see, for instance, \cite{Parlett1998}).

For $p>2$, computing $C_{e,h}(p)$ exactly requires solving a
nonlinear constrained maximization problem. In practice we use
the approximation
\[
\widehat C_{e,h}(p)
=
\frac{\|\phi_1\|_{p,h}^p}{\lambda_1^{p/2}},
\]
where $\phi_1$ is normalized by $\phi_1^\top W\phi_1=1$.
Although this quantity is not necessarily an upper bound for
$C_{e,h}(p)$, it provides a simple and inexpensive estimate that
is sufficient for verifying the numerical threshold condition
used in Section \ref{sec7}.


\begin{thebibliography}{23}

\bibitem{Akil1}
M. Akil, Y. Chitour, M. Ghader, A. Wehbe, Stability and exact controllability of a timoshenko system with only one fractional damping on the boundary, {\it Asymptot. Anal.}, \textbf{119} (2020), 221--280.

\bibitem{Akil2}
M. Akil, A. Wehbe, Stabilization of multidimensional wave equation with locally boundary fractional dissipation law under geometric conditions, {\it Math. Control Relat. Fields}, \textbf{9} (2019), 97--116.



\bibitem{Algwaiz}
M. Al-Gwaiz, V. Benci, F. Gazzola, Bending and stretching energies in a rectangular plate modeling suspension bridges, {\it Nonlinear Anal.}, \textbf{106} (2014), 181--734.

\bibitem{ammari25}
K. Ammari, V. Komornik, M. Sepúlveda-Cortés, O. Vera-Villagr\'an, Numerical Stabilization for a Mixture System with Kind Damping, {\it Appl. Math. Optim.}, \textbf{92} (2025), 57.

\bibitem{Benaissa2}
R. Aounallah, A. Benaissa, A. Zara\"{i}, Blow-up and asymptotic behavior for a wave equation with a time delay condition of fractional type, {\it Rend. Circ. Mat. Palermo}, \textbf{70} (2021), 1061--1081.

\bibitem{FA2}
M. F. Aslam, J. Hao, Nonlinear logarithmic wave equations: Blow-up phenomena and the influence of fractional damping, infinite memory and strong dissipation, {\it Evol. Equ. Control Theory}, \textbf{13} (2024), 1423--1435.

\bibitem{FA1}
M. F. Aslam, J. Hao, S. Boulaaras, L. Bashir, Blow-Up of Solutions in a Fractionally Damped Plate Equation with Infinite Memory and Logarithmic Nonlinearity, {\it Axioms}, \textbf{14} (2025), 80.


\bibitem{FA3}
M. F. Aslam, J. Hao, Z. Hajjej, L. Bashir, On the global existence, exponential decay and blow-up of a nonlinear wave equation subject to a boundary fractional damping and time-varying delay, {\it Discrete Contin. Dyn. Syst. Ser. S}, (2025).

\bibitem{aslam25}
M. F. Aslam, J. Hao, M. Sepúlveda, Z. Hajjej, Dynamic of Logarithmically and Fractionally Damped Wave Equations With Strong Damping and Infinite Memory: Theory and Numerics, {\it Math. Methods Appl. Sci.}, \textbf{48} (2025), 16617--16627.

\bibitem{Benaissa1}
A. Benaissa, H. Benkhedda, Global existence and energy decay of solutions to a wave equation with a dynamic boundary dissipation of fractional derivative type, {\it Z. Anal. Anwend.}, \textbf{37} (2018), 315--339.

\bibitem{Berchio}
E. Berchio, A. Ferrero, F. Gazzola, Structural instability of nonlinear plates modelling suspension bridges: mathematical answers to some long-standing questions, {\it Nonlinear Anal. Real World Appl.}, \textbf{28} (2016), 91--125.

\bibitem{BC}
E. Blanc, G. Chiavassa, B. Lombard, Biot-JKD model: Simulation of 1D transient poroelastic waves with fractional derivative, {\it J. Comput. Phys.}, \textbf{237} (2013), 1--20.

\bibitem{Vuk}
I. Bochicchio, C. Giorgi, E. Vuk, Asymptotic dynamics of nonlinear coupled suspension bridge equations, {\it J. Math. Anal. Appl.}, \textbf{402} (2013), 319--333.

\bibitem{Carpenter1999}
M. H. Carpenter, J. Nordstr\"om, D. Gottlieb, A Stable and Conservative Interface Treatment of Arbitrary Spatial Accuracy, {\it J. Comput. Phys.}, \textbf{148} (1999), 341--365.

\bibitem{CM}
J. U. Choi, R. C. Maccamy, Fractional order Volterra equations with applications to elasticity, {\it J. Math. Anal. Appl.}, \textbf{139} (1989), 448--464.

\bibitem{DA}
M. D'Abbicco, L. G. Longen, The interplay between fractional damping and nonlinear memory for the plate equation, {\it Math. Methods Appl. Sci.}, \textbf{45} (2022), 6951--6981.

\bibitem{delfour}
M. Delfour, M. Fortin, G. Payr, Finite difference solutions of a non-linear Schrödinger equation, {\it J. Comput. Phys.}, \textbf{44} (1981), 277--288.

\bibitem{Fernandez2014review}
D. C. Del Rey Fern\'andez, J. E. Hicken, D. W. Zingg, Review of Summation-By-Parts Operators with Simultaneous Approximation Terms for the Numerical Solution of Partial Differential Equations, {\it Comput. Struct.}, 2014.

\bibitem{Ferr1}
A. Ferrero, F. Gazzola, A partially hinged rectangular plate as a model for suspension bridges, {\it Discrete Contin. Dyn. Syst. A}, \textbf{35} (2015), 5879--5908.



\bibitem{Gazzola}
F. Gazzola, {\em Mathematical Models for Suspension Bridges: Nonlinear Structural Instability, Modeling, Simulation and Applications}, Springer-Verlag, New York, 2015.

\bibitem{Meck1}
J. Glover, A. C. Lazer, P. J. McKenna, Existence and stability of large scale nonlinear oscillation in suspension bridges, {\it Z. Angew. Math. Phys.}, \textbf{40} (1989), 172--200.

\bibitem{Haj3}
Z. Hajjej, General decay of solutions for a viscoelastic suspension bridge with nonlinear damping and a source term, {\it Z. Angew. Math. Phys.}, \textbf{72} (2021), 72--90.

\bibitem{ZH}
Z. Hajjej, A suspension bridges with a fractional time delay: asymptotic behavior and blow-up in finite time, {\it AIMS Math.}, \textbf{9} (2024), 22022--22040.

\bibitem{Haj2}
Z. Hajjej, M. M. Al-Gharabli, S. A. Messaoudi, Stability of a suspension bridge with a localized structural damping, {\it Discrete Contin. Dyn. Syst. Ser. S}, \textbf{15} (2022), 1165--1181.

\bibitem{Haj1}
Z. Hajjej, S. A. Messaoudi, Stability of a suspension bridge with structural damping, {\it Ann. Polon. Math.}, \textbf{125} (2020), 59--70.

\bibitem{KR}
M. Kirane, R. Aounallah, L. Jlali, General Decay and Blowing‐Up Solutions of a Nonlinear Wave Equation With Nonlocal in Time Damping and Infinite Memory, {\it Math. Methods Appl. Sci.}, 2025.

\bibitem{krenk}
S. Krenk, Energy conservation in Newmark based time integration algorithms, {\it Comput. Methods Appl. Mech. Engrg.}, \textbf{195} (2006), 6110--6124.



\bibitem{Liu}
W. Liu, H. Zhuang, Global existence, asymptotic behavior and blow-up of solutions for a suspension bridge equation with nonlinear damping and source terms, {\it Nonlinear Differ. Equ. Appl.}, \textbf{24} (2017).

\bibitem{Ma}
Q. Ma, C. Zhong, Existence of strong solutions and global attractors for the coupled suspension bridge equations, {\it J. Differential Equations}, \textbf{246} (2009), 1003--1014.

\bibitem{BM}
B. Mbodje, Wave energy decay under fractional derivative controls, {\it IMA J. Math. Control Inf.}, \textbf{23} (2006), 237--257.

\bibitem{MainBo}
M. Mainardi, E. Bonetti, The application of real-order derivatives in linear viscoelasticity, in H. Giesekus and M. F. Hibberd (eds.), {\it Progress and Trends in Rheology II}, Steinkopff, Heidelberg, (1988), 64--67.

\bibitem{Mattsson2004}
K. Mattsson, Summation by parts operators for finite difference approximations of second derivatives, {\it J. Comput. Phys.}, \textbf{199} (2004), 503--540.

\bibitem{Meck2}
P. J. McKenna, W. Walter, Nonlinear oscillations in a suspension bridge, {\it Arch. Ration. Mech. Anal.}, \textbf{98} (1987), 167--177.

\bibitem{Mess1}
S. A. Messaoudi, S. E. Mukiawa, A Suspension Bridge Problem: Existence and Stability, Springer Proc. Math. Stat., 190 (2017), 151--165.

\bibitem{Mess2}
S. A. Messaoudi, S. E. Mukiawa, Existence and decay of solutions to a viscoelastic plate equations, {\it Electron. J. Differential Equations}, \textbf{2016} (2016), 1--14.

\bibitem{MU}
S. E. Mukiawa, Decay result for a delay viscoelastic plate equation, {\it Bull. Braz. Math. Soc.}, \textbf{51} (2020), 333--356.

\bibitem{RV2}
P. X. Pamplona, J. E. M. Rivera, R. Quintanilla, On the decay of solutions for porous-elastic systems with history, {\it J. Math. Anal. Appl.}, \textbf{379} (2011), 682--705.

\bibitem{Parlett1998}
B. N. Parlett, {\em The Symmetric Eigenvalue Problem}, SIAM Classics in Applied Mathematics, 1998.

\bibitem{PZ}
A. Pazy, {\em Semigroups of Linear Operators and Applications to Partial Differential Equations}, Springer, New York, 1983.

\bibitem{Pod}
I. Podlubny, {\em Fractional Differential Equations: An Introduction to Fractional Derivatives, Fractional Differential Equations, to Methods of their Solution and some of their Applications}, Academic Press, London, 1999.

\bibitem{RV1}
J. E. M. Rivera, H. D. F. Sare, Stability of Timoshenko systems with past history, {\it J. Math. Anal. Appl.}, \textbf{339} (2008), 482--502.

\bibitem{RosSh}
Y. Rossikhin, M. Shitikova, Application of Fractional Calculus for Analysis of Nonlinear Damped Vibrations of Suspension Bridges, {\it J. Eng. Mech.}, \textbf{124} (1998), 1029--1036.

\bibitem{Tazhimbetov2023}
N. Tazhimbetov et al., Simulation of flexural-gravity wave propagation for elastic plates in shallow water using an energy-stable finite difference method with weakly enforced boundary and interface conditions, {\it J. Comput. Phys.}, \textbf{493} (2023), 112470.


\bibitem{TorBa}
P. J. Torvik, R. L. Bagley, On the appearance of the fractional derivative in the behavior of real materials, {\it J. Appl. Mech.}, \textbf{51} (1984), 294--298.


\bibitem{Wang}
Y. Wang, Finite time blow-up and global solutions for fourth-order damped wave equations, {\it J. Math. Anal. Appl.}, \textbf{418} (2014), 713--733.

\end{thebibliography}
\end{document}